\documentclass[11pt,reqno]{amsart}
\usepackage{amsfonts}
\usepackage{amsmath}
\usepackage{amssymb}
\usepackage{multicol}
\usepackage{stmaryrd}
\usepackage{cite}
\usepackage{epsfig}
\usepackage{color}
\usepackage{graphics}
\usepackage{graphicx}
\usepackage{booktabs}
\usepackage{epstopdf}
\usepackage{multicol,graphics}
\usepackage{subfigure}
\newcommand\bes{\begin{eqnarray}}
\newcommand\ees{\end{eqnarray}}

\newtheorem{theorem}{Theorem}[section]
\newtheorem{lemma}[theorem]{Lemma}
\newtheorem{corollary}[theorem]{Corollary}

\newtheorem{definition}[theorem]{Definition}
\newtheorem{remark}[theorem]{Remark}

\numberwithin{equation}{section}
\allowdisplaybreaks

\allowdisplaybreaks[4]

\usepackage{hyperref}
\hypersetup{hidelinks,
colorlinks=true,linkcolor=blue,citecolor=red,
pdfstartview=Fit,
pdfborder={1 1 1},
bookmarksnumbered=true,
breaklinks=true}

\begin{document}

\title[A Hamilton-Jacobi Approach to Time-Delayed Nonlocal Diffusion Models]{A Hamilton-Jacobi Approach to Time-Delayed Nonlocal Diffusion Models in Shifting Habitats}

\author[Jiang, Lam, Li and Tao]{Bing-Er Jiang$^{1}$, King-Yeung Lam$^{2}$,  Wan-Tong Li$^{1,*}$ and Wen Tao$^{1}$}
\thanks{\hspace{-.6cm}
$^1$School of Mathematics and Statistics, Lanzhou University, Lanzhou, Gansu 730000, P.R. China.
\\
$^2$Department of Mathematics, The Ohio State University, Columbus, OH, 43210, USA\\
$^*${\sf Corresponding author} (wtli@lzu.edu.cn)}

\date{\today}

\begin{abstract}
This paper is concerned with the spatial propagation dynamics of time-delayed nonlocal diffusion equations in shifting habitats. We employ the theory of viscosity solutions for Hamilton-Jacobi equations to provide a complete classification of the spreading speeds. In particular, we derive variational formulas that explicitly characterize how these speeds depend on the decay rate of the initial data, the habitat shifting speed, and the maturation delay across three regimes: locally determined, nonlocally selected, and locked. We reveal a distinct directional asymmetry in the nonlocal selection mechanism and a novel delay-insensitivity phenomenon. We also establish the threshold conditions under which the habitat locking effect eliminates the classical decelerating role of time delay.

Keywords: spreading speed, time-delay, reaction-diffusion, viscosity solution

\textbf{AMS Subject Classification (2020)}: 35K57, 35R20, 92D25.

\end{abstract}

\maketitle

\tableofcontents

\newpage

\section{Introduction}\label{Sec:Int}

A standard model for population spread with nonlocal dispersal and maturation delay takes the form
\begin{equation}\label{eq:base_model}
    \dfrac{\partial u(t,x)}{\partial t}=d\mathcal{D}[u](t,x)-\mu u(t,x)+ \mu g(u(t-\tau,x)), \quad t>0, \, x\in\mathbb{R},
\end{equation}
which has been widely studied \cite{Fang2014,Yi2015,Liu2015,Li2020}.
Here, $u(t,x)$ denotes the density of mature individuals, $d>0$ is the dispersal rate, $\mu>0$ is the mortality rate, $\tau\ge0$ is the maturation delay. The nonlocal dispersal operator $\mathcal{D}$ is defined by
\begin{equation*}
\mathcal{D}[u](t,x)=\int_{\mathbb{R}}J(x-y)u(t,y)\,dy-u(t,x),
\end{equation*}
where $J$ is a continuous probability density on $\mathbb{R}$ describing both short-range and long-distance individual movements \cite{Turchin,Mollison1977,Murray2003,Coville2007,Ignat2007,Shen2010}.
This general formulation encompasses several classical biological models, most notably the Nicholson's blowflies equation with nonlocal dispersal \cite{Yi2015,Liu2015,Pan2010}, where the birth function takes the form
\[
g(u) = \alpha u e^{-\delta u}
\]
with $\alpha$ the maximum per capita egg production rate and $1/\delta$ the population size at which reproduction is maximized.

While model \eqref{eq:base_model} characterizes propagation in a static environment, it does not account for the effects of global climate change, which is causing systematic shifts in species' habitats \cite{Potapov2004}.
The mathematical study in shifting habitats was pioneered by Potapov and Lewis \cite{Potapov2004}, and the canonical model now widely used in the literature was introduced by Berestycki et al. \cite{Berestycki2009}:
\begin{align}\label{eq:local-kpp}
    \frac{\partial u(t,x)}{\partial t}=\frac{\partial^2 u(t,x)}{\partial x^2}+u(t,x)\left[\alpha(x-c_{1}t)-u(t,x)\right], \quad  (t,x)\in(0,\infty)\times\mathbb{R}.
\end{align}
Here, the growth rate depends explicitly on the moving coordinate $z = x - c_1 t$, where $c_1$ is the speed at which the habitat shifts. A fundamental mathematical challenge posed by this formulation is the loss of spatial translation invariance, which invalidates many classical techniques developed for homogeneous systems. Despite this difficulty, significant progress has been made in understanding the propagation dynamics of \eqref{eq:local-kpp} and its variants over the past decade \cite{Fang2016, Hu2020, Li2014,Lam2022,Lam2025}.

Integrating this environmental forcing into the delayed nonlocal framework \eqref{eq:base_model} yields the primary subject of this paper:
\begin{equation}\label{eq:main}
\begin{cases}
    \dfrac{\partial u(t,x)}{\partial t}=d\mathcal{D}[u](t,x)-\mu u(t,x)+ \mu f(x-c_{1}t,u(t-\tau,x)),&(t,x)\in(0,\infty)\times\mathbb{R}, \\
    u(\theta,x)=\phi(\theta,x), & (\theta,x)\in[-\tau,0]\times\mathbb{R}.
\end{cases}
\end{equation}
By formally taking the Dirac delta distribution kernel $J(x) = \delta(x) + \delta''(x)$, the nonlocal operator reduces to the Laplacian, and system \eqref{eq:main} recovers the corresponding local reaction-diffusion equation with shifting habitats:
\begin{align}\label{eq:local}
\dfrac{\partial u(t,x)}{\partial t}=d\Delta u-\mu u(t,x)+ \mu f(x-c_{1}t,u(t-\tau,x)),\quad(t,x)\in(0,\infty)\times\mathbb{R}.
\end{align}
While our primary focus is on nonlocal diffusion, we note that equation \eqref{eq:local} can be analyzed using a similar approach. See Remark \ref{remark:local-new} for details.

The central question we address in this paper is: under what conditions can a population successfully track its shifting habitat, and what determines the speed at which it spreads? To answer this question, we use the concept of spreading speeds, which provide a quantitative measure of invasion success. Following Aronson and Weinberger \cite{Aronson1978}, we say that a nonnegative solution $u$ admits rightward and leftward spreading speeds $c_u^r$ and $c_u^l$ if
\begin{align*}
\begin{cases}
\lim\limits_{t\to+\infty} \sup\limits_{x\le (-c_u^l-\eta)t,\text{or }x\ge(c_u^r+\eta)t}u(t,x)=0 \quad &\text{ for all } \eta>0,\\
\liminf\limits_{t\to+\infty} \inf\limits_{ (-c_u^l+\eta)t\le x\le(c_u^r-\eta)t}u(t,x)>0 \quad &\text{ for all } \eta\in(0,(c_u^r+c_u^l)/2).
\end{cases}
\end{align*}

The dynamics of shifting habitat models heavily depend on the asymptotic behavior of the growth function at spatial infinity. We define
\begin{align*}
 f^{\pm}(u) := \lim_{z \to \pm \infty} f(z,u),
\end{align*}
which characterize the habitat conditions in the far field. Based on these limits, habitats can be classified into three types: Bad-Bad (BB), Bad-Good (BG), and Good-Good (GG), where a ``good" habitat is one in which the corresponding homogeneous system admits a positive spreading speed, and a ``bad" habitat is one in which the population goes extinct. While significant progress has been made for BB-type \cite{Yi2023,Chen2025} and BG-type \cite{Yi2020, Yi2025-2} habitats, GG-type habitats, where conditions are favorable at both infinities, remain the least understood due to the emergence of complex propagation phenomena such as nonlocal speed selection and front locking.

Most recently, Yi and Zhao \cite{Yi2025} advanced the semiflow framework for GG-type habitats with local diffusion. A notable strength of their work is that it allows for non-monotone reaction terms, making it applicable to a wider class of biological models. However, three key limitations motivate our study. First, their spreading speed analysis is only valid for $-\min\{c_-^*,c_+^*\}<c_1<\min\{c_-^*,c_+^*\}$, where $c_\pm^*$ denote the far-field homogeneous spreading speeds; all other regimes remain only partially resolved. Second, their methods do not readily extend to nonlocal diffusion, as the absence of regularizing effects from the Laplacian makes verifying the compactness assumptions required by the semiflow approach highly nontrivial \cite{Li2018}. Third, their approach cannot explicitly quantify how spreading speeds depend on the initial exponential decay rate, which is essential for understanding nonlocal selection phenomena. Table \ref{tab:literature_compare} provides a systematic comparison of our work with existing literature.
\begin{table}[ht]
\centering
\caption{Comparison of existing results and this work.}\label{tab:literature_compare}
\renewcommand{\arraystretch}{1.2}
\begin{tabular}{l c c c c}
\toprule
Reference & Diffusion Operator & Habitat Type & Monotonicity of $f$ & Time Delay \\
\midrule
\cite{Yi2023} & Local & BB & Non-monotone & Discrete \\
\cite{Chen2025} & Nonlocal & BB & Non-monotone & Discrete \\
\cite{Yi2020} & Nonlocal \& Local & BG & Monotone & Discrete \\
\cite{Yi2025-2} & Local & BG & Non-monotone & Discrete \\
\cite{Yi2025} & Local & GG & Non-monotone & Discrete \\
\cite{Lam2022} & Local & GG & Weakly monotone & Distributed \\
\midrule
This work & Nonlocal \& Local$^\dagger$ & GG & Monotone & Discrete \\
\bottomrule
\end{tabular}
\begin{flushleft}
\small{$^\dagger$ Analogous results hold for the local diffusion case; see Remarks \ref{remark:local-refs} and \ref{remark:local-new} for details.}
\end{flushleft}
\end{table}

To overcome these limitations, we adopt the Hamilton-Jacobi (H-J) approach, which has emerged as a powerful alternative for studying propagation problems. Originating from the work of Freidlin \cite{freidlin1985} and Evans and Souganidis \cite{evans1989}, this methodology characterizes the asymptotic spreading behavior via the viscosity solution of a macroscopic H-J equation, thereby bypassing the need for the strict compactness and regularity assumptions required by semiflow theory.

The H-J framework has been widely used to study propagation dynamics in complex systems, including Lotka-Volterra competition systems \cite{liu2020,Liu2021}, nonlocal equations in almost periodic media \cite{Liang2020}, age-structured models \cite{kang2025}, and kinetic equations \cite{loy2024}. We refer the readers to \cite{Lam2026} for more details. Recently, this approach has profoundly advanced the study of heterogeneous problems. For classical reaction-diffusion equations in shifting habitats with distributed delays, Lam and Yu \cite{Lam2022} gave a complete characterization of spreading speeds in weakly monotonic GG-type habitats, and quantified how the speed depends on the exponential decay rate of initial data.
Later, Lam, Nadin, and Yu \cite{Lam2025} extended their results to equation \eqref{eq:local-kpp} with more general intrinsic growth rate function and identified two distinct regimes: nonlocally selected fronts \cite{girardin2019} and locked fronts. For nonlocal dispersal, Tao et al. \cite{Tao2026} applied this approach to Fisher-KPP equations in shifting habitats and discovered a directional asymmetry for the occurrence of nonlocal speed selection in spatially increasing versus decreasing environments.

Motivated by these developments, we adopt the H-J framework to rigorously investigate the time-delayed nonlocal system \eqref{eq:main}. In our study, this methodological shift is paired with slightly stronger structural assumptions (e.g., monotonicity). This strategy enables us to move beyond the aforementioned partial characterizations. Unlike the results in \cite{Yi2025}, our analysis derives explicit variational formulas and a complete classification of spreading speeds for all \(c_1\in\mathbb{R}\), covering both exponentially decaying initial data \rm{(IC\(^\lambda\))} and compactly supported initial data \rm{(IC\(^\infty\))} (see Theorems \ref{thm:right}--\ref{thm:vis-infty}).
Our analysis provides a \textbf{complete classification} of the spreading speeds for the nonlocal model, yielding a phase diagram (see Figure \ref{fig:images}) that describes the parameter regimes governing locally determined, nonlocally selected, and locked propagation. Furthermore, all classification regimes and explicit formulas apply to the local setting (see Remark \ref{remark:local-new}). The specific analysis for the Laplacian case is deferred to Appendix \ref{App:Local}, thereby yielding a comprehensive description that strictly improves upon existing partial results for local models \cite{Yi2025}.

Our primary novelties lie in two aspects:
\begin{itemize}
\item \textbf{Nonlocal speed selection in time-delayed systems:} We establish the existence of a nonlocal speed selection mechanism in time-delayed nonlocal reaction-diffusion equations with shifting habitats. We show that the environmental heterogeneity forces the spreading speed to be determined by a global coupling between the near-field information of the propagation front and the far-field habitat conditions rather than just local information at the front. This is manifested as a dynamic adjustment of the decay rate driven by the heterogeneous distribution of $f(x-c_{1}t,\cdot)$. Furthermore, we prove a strict directional asymmetry in the speed selection mechanism, extending the recent results of Tao et al. \cite{Tao2026} to time-delayed systems: for spatially increasing habitats, nonlocal selection determines the leftward speed only for sufficiently small initial decay rates, but dominates the rightward speed even for compactly supported initial data.

\item \textbf{Delay-independence and locking thresholds:} We provide a precise quantitative analysis of how maturation delay affects spreading speeds in shifting habitats. In homogeneous environments, it is well known that delay always decelerates propagation speed \cite{Pan2009,Schaaf1987,Wang2008}. In sharp contrast, we demonstrate that in GG-type shifting habitats, this monotonic deceleration effect is completely disrupted by the locking phenomenon. Once the propagation front locks to the habitat shifting speed, the spreading speed becomes independent of the delay, remaining exactly equal to $c_1$ regardless of further increases in $\tau$. We explicitly compute the critical delay thresholds that mark the transition to and from the locking regime.
\end{itemize}

\textbf{Outline of the paper:} Section \ref{Sec:Pre} contains the key assumptions, and presents the main results, including the explicit formulas for the spreading speeds. Section \ref{Sec:HJ} is devoted to the derivation of the limiting Hamilton-Jacobi equation. Section \ref{Sec:Speed} provides the explicit formulas for the spreading speeds and completes the proof of main theorems. Section \ref{Sec:tau} establishes the effect of the time delay on the spreading speeds. Section \ref{Sec:Num} presents numerical simulations to verify the theoretical findings.

\section{Preliminary}\label{Sec:Pre}

In this section, we present the main hypotheses and results. Let \( X = BUC(\mathbb{R}, \mathbb{R}) \) be the normed vector space of all bounded and uniformly continuous functions from \( \mathbb{R} \) to \( \mathbb{R} \) with the usual supremum norm. Let $X^+=BUC(\mathbb{R},\mathbb{R}^+)$.
Let \( C = C([-\tau, 0], X) \) be the normed vector space of all continuous functions from \( [-\tau, 0] \) into \( X \) with the norm $\|\phi\|_C := \sup\limits_{\theta \in [-\tau,0]} \|\phi(\theta)\|_X$ and $C^+=C([-\tau, 0], X^+)$. We shall also treat an element $\phi\in C$ as a function from $[-\tau,0]\times\mathbb{R}$ into $\mathbb{R}$, so that $\phi\in BUC([-\tau,0]\times\mathbb{R},\mathbb{R})$. Let \(C_{r}=\{\phi\in BUC([-\tau,0]\times\mathbb{R},\mathbb{R}): 0\le \phi(\theta,x)\le r\}\) for some $r>0$.

We begin with the assumption on the dispersal kernel $J$.
\begin{itemize}
\item[(J)] $J\in C(\mathbb{R})$ is nonnegative, symmetric and compactly supported and $\int_\mathbb{R}J(y)dy=1$.
\end{itemize}
Next state the main hypotheses on the nonlinearity $f$.
\begin{enumerate}
    \item[(F1)] $f(z, 0) = 0$ for all $z\in\mathbb{R}$, \(f(z,u)>0\) for \((z,u)\in\mathbb R\times(0,\infty)\), $f \in C(\mathbb{R} \times [0, \infty))$ and $f(z,\cdot)\in C^1([0,\infty))$. Moreover,
          \begin{align}\label{eq:F2-1}
           f(z, u) \le u \, \partial_u f(z, 0) \quad \text{for all } (z, u),
          \end{align}
          for each $M > 0$, $\sup\limits_{\mathbb{R} \times [0,M]} |\partial_u f(z, u)| < \infty$. For each $\eta_*>0$, there exists $\delta_*>0$ such that
          \[
          f(z, u) \ge \left( \partial_u f(z, 0) - \eta_* \right) u \qquad \text{whenever } 0 \le u \le \delta_*.
          \]

    \item[(F2)] The limits $f^{\pm}(u):=\lim\limits_{z\to\pm\infty}f(z,u)$ exist in $C_{\text{loc}}^1(\mathbb{R}^+,\mathbb{R}^+)$ and $\frac{\mathrm{d} f^\pm(0)}{\mathrm{d} u}>1$.

    \item[(F3)] The function $\partial_u f(z,0)$ is strictly nondecreasing in $z\in\mathbb{R}$.

    \item[(F4)] The function $f(z,u)$ is nondecreasing in $u\in[0,u^*]$, where $u^*$ is given by the following condition: there exists $u^*>0$ such that $f^+(u^*)=u^*$ and
    \[
    (f^+(u)-u)(u-u^*)<0 \quad\text{ for all } u\in(0,\infty)\backslash\{u^*\}.
    \]
\end{enumerate}
A typical form of $f$ is Nicholson's blowflies model, that is $f(z,u)=\alpha(z)ue^{-u}$, where $\alpha(z)\in C(\mathbb{R})$ is non-decreasing, $1<\alpha(z)<e$, and $\alpha(\pm\infty)$ exists.

Define \( R, \underline{R} \in L^\infty(\mathbb{R}) \) be given by
\begin{align*}
\begin{cases}
R(x/t) := \limsup\limits_{\substack{\varepsilon \to 0 \\ (t',x') \to (t,x) }} \partial_u f(\frac{x'-c_1t'}{\varepsilon}, 0) & \text{for } (t,x) \in (0, \infty) \times \mathbb{R}, \\
\underline{R}(x/t) := \liminf\limits_{\substack{ \varepsilon \to 0 \\ (t',x') \to (t,x) }} \partial_u f(\frac{x'-c_1t'}{\varepsilon}, 0) & \text{for } (t,x) \in (0, \infty) \times \mathbb{R}.
\end{cases}
\end{align*}
As a consequence of (F2) and (F3), there exist two constants \(R_\pm=\frac{\mathrm{d}f^\pm(0)}{\mathrm{d}u}>1\) such that
\begin{align}\label{eq:R+-}
R(s)=\begin{cases}
R_+, & s\ge c_1,\\[2pt]
R_-, & s<c_1,
\end{cases}\qquad
\underline{R}(s)=\begin{cases}
R_+, & s>c_1,\\[2pt]
R_-, & s\le c_1.
\end{cases}
\end{align}
In particular, we have
\begin{align}\label{under-R}
R(s)=\underline{R}(s) \quad\text{for a.e. } s\in\mathbb{R} \quad\text{and }\quad R(s)\geqslant \underline{R}(s)>1 \quad\text{for all }  s\in\mathbb{R}.
\end{align}

We now state the assumptions on the initial data $\phi: [-\tau, 0] \times \mathbb{R} \to \mathbb{R}$.
\begin{itemize}
\item[$(\text{IC}^\lambda)$] $\phi \in C_{u^*}$ is positive, and there exist constants $0 < a < b$, $\lambda^r > 0$, $\lambda^l > 0$ such that
\begin{align*}
&ae^{-(\lambda^r+o(1))x} \leq \phi(\theta,x) \leq be^{-(\lambda^r+o(1))x} \quad \text{for } x \gg 1, \forall \theta \in [-\tau,0],\\
&ae^{(\lambda^l+o(1))x} \leq \phi(\theta,x) \leq be^{(\lambda^l+o(1))x} \quad \text{for } x \ll -1, \forall \theta \in [-\tau,0].
\end{align*}
\item[$(\text{IC}^\infty)$] $\phi \in C_{u^*}$ is non-trivial, nonnegative and decays faster than any exponential, i.e.,
\[
\lim_{|x|\to+\infty}{e^{\lambda |x|}}\sup_{\theta\in[-\tau,0]}\phi(\theta,x)=0 \quad\text{ for every } \lambda>0.
\]
\end{itemize}
Notably, compactly supported initial data satisfies $(\text{IC}^\infty)$. In what follows, we will sometimes use the phrase ``$(\text{IC}^\lambda)$ holds with $\lambda = +\infty$" to indicate that condition $(\text{IC}^\infty)$ is in effect.

For any initial datum \(\phi\in C_{u^*}\), system \eqref{eq:main} admits a unique solution on \(\mathbb{R}^{+}\), denoted by \((u^{\phi;f})_{t}\). This solution is classical for \(t>0\), and the map $t \longmapsto (u^{\phi;f})_{t}\in C_{u^*}$ is continuous. Here \((u^{\phi;f})_{t}(\theta,x)=u^{\phi;f}(t+\theta,x)\) for all \((\theta,x)\in[-\tau,0]\times\mathbb{R}\). These conclusions follow from the standard step method in \cite[Theorem 5.1]{Liu2015}.

Below we introduce several notions related to the Hamilton-Jacobi framework. Let us first introduce the Hamiltonian appearing in the constrained Hamilton-Jacobi equation that will be used to characterize the propagation phenomena. For real variables $s,p,q$, define the Hamiltonian $H:\mathbb{R}\times\mathbb{R}\times\mathbb{R}\to\mathbb{R}$ by
\begin{align}\label{H}
H(s,p,q)=d\left(\int_{\mathbb{R}}J(y)e^{py}dy-1\right) + q-\mu+\mu R(s)e^{\tau q},
\end{align}
where $R(\cdot)$ is explicitly given by \eqref{eq:R+-}.
Note that for each fixed $s$, the function $(p,q)\mapsto H(s,p,q)$ is convex. Consider the constrained Hamilton-Jacobi equation
\begin{equation}\label{eq:constrained-HJ}
\min\left\{w,H(x/t,\partial_x w,\partial_t w)\right\}=0,
\qquad (t,x)\in(0,+\infty)\times\mathbb{R}.
\end{equation}
We recall the notion of viscosity subsolutions and supersolutions for \eqref{eq:constrained-HJ}, see e.g. \cite{Ishii2013}.
\begin{definition}
A locally bounded lower semi-continuous function $w$ is a viscosity supersolution of \eqref{eq:constrained-HJ} if $w \geq 0$, and for all $\varphi \in C^1$, if $(t_0,x_0)$ is a strict local minimum of $w - \varphi$, then
\[
H^*(x_0/t_0,\partial_x\varphi(t_0,x_0),\partial_t\varphi(t_0,x_0)) \geq 0.
\]
A locally bounded upper semi-continuous function $w$ is a viscosity subsolution of \eqref{eq:constrained-HJ} if for all $\varphi \in C^1$, if $(t_0,x_0)$ is a strict local maximum of $w - \varphi$ and $w(t_0,x_0) > 0$, then
\[
H_*(x_0/t_0,\partial_x\varphi(t_0,x_0),\partial_t\varphi(t_0,x_0)) \leq 0.
\]
Finally, $w$ is a viscosity solution of \eqref{eq:constrained-HJ} if and only if $w$ is a viscosity supersolution and subsolution.
Here, $H^*$ and $H_*$ are the semi-continuous envelope of $H$. Precisely,
\[
H^*(s,p,q)=\limsup_{(s',p',q')\to (s,p,q)}H(s',p',q') \quad \text{and} \quad H_*(s,p,q)=\liminf_{(s',p',q')\to (s,p,q)}H(s',p',q').
\]
\end{definition}

For fixed $\tau\ge0$, let $\mathbf{H}_\pm(p):\mathbb{R}\to(0,\infty)$ be uniquely defined by
\begin{align}\label{gamma}
\Delta_\pm(p,\mathbf{H}(p)):=d\left(\int_{\mathbb{R}}J(y)e^{py}dy-1\right) - \mathbf{H}(p) -\mu+\mu R_\pm e^{-\tau \mathbf{H}(p)}=0,
\end{align}
where $\mathbf{H}_\pm$ is strictly convex, a property proven in Section \ref{Sec:Speed}.
Let $\Psi_+ : \mathbb{R} \to \mathbb{R}$ denote the inverse of $\mathbf{H}_+' : \mathbb{R} \to \mathbb{R}$.
Define the Lagrangian $\mathbf{L}_+(v)$ by
\[
\mathbf{L}_+(v) = \max_{p\in\mathbb{R}} [vp - \mathbf{H}_+(p)] \quad \text{for } v\in\mathbb{R},
\]
which, by the first-order condition, can be written as $\mathbf{L}_+(v) = [vp - \mathbf{H}_+(p)]_{p=\Psi_+(v)}$.
Set $c_\pm(p)=\frac{\mathbf{H}_\pm(p)}{p}$. We then define the pair $(c_+^*, \nu_+^*)$ and $(c_-^*, \nu_-^*)$ via $c_\pm^* = \inf\limits_{p>0} c_\pm(p)=c_\pm(\nu_\pm^*)$.
See details in Section \ref{Sec:Speed}.

Now we can state the main theorem of the paper. The following two theorems establish the explicit formulas for the spreading speeds when  (\text{IC}$^\lambda$) holds.
\begin{figure}[htpb]
	\centering
	\subfigure[Division for rightward propagation.]{\includegraphics[width=0.48\textwidth]{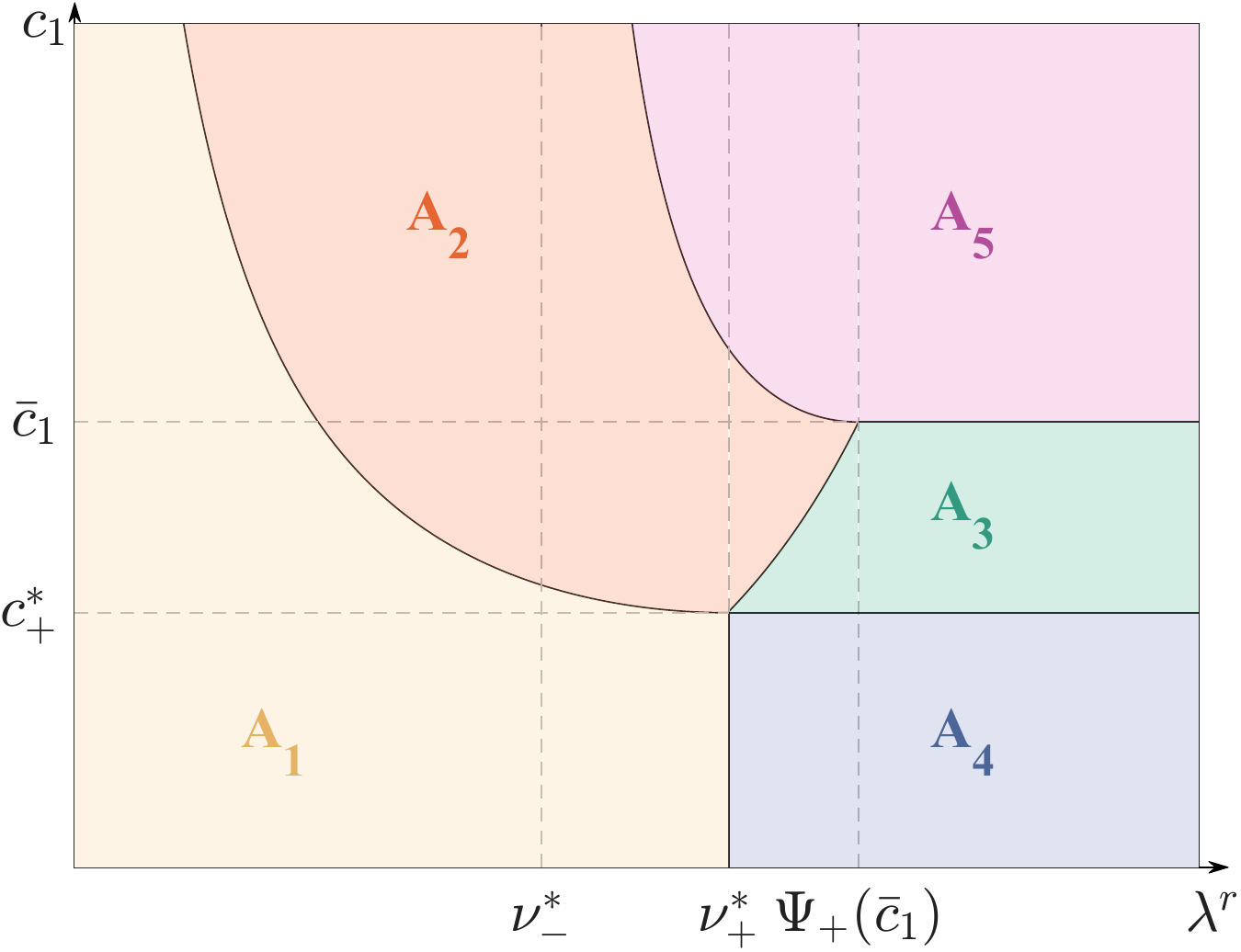}}
	\subfigure[Division for leftward propagation.]{\includegraphics[width=0.49\textwidth]{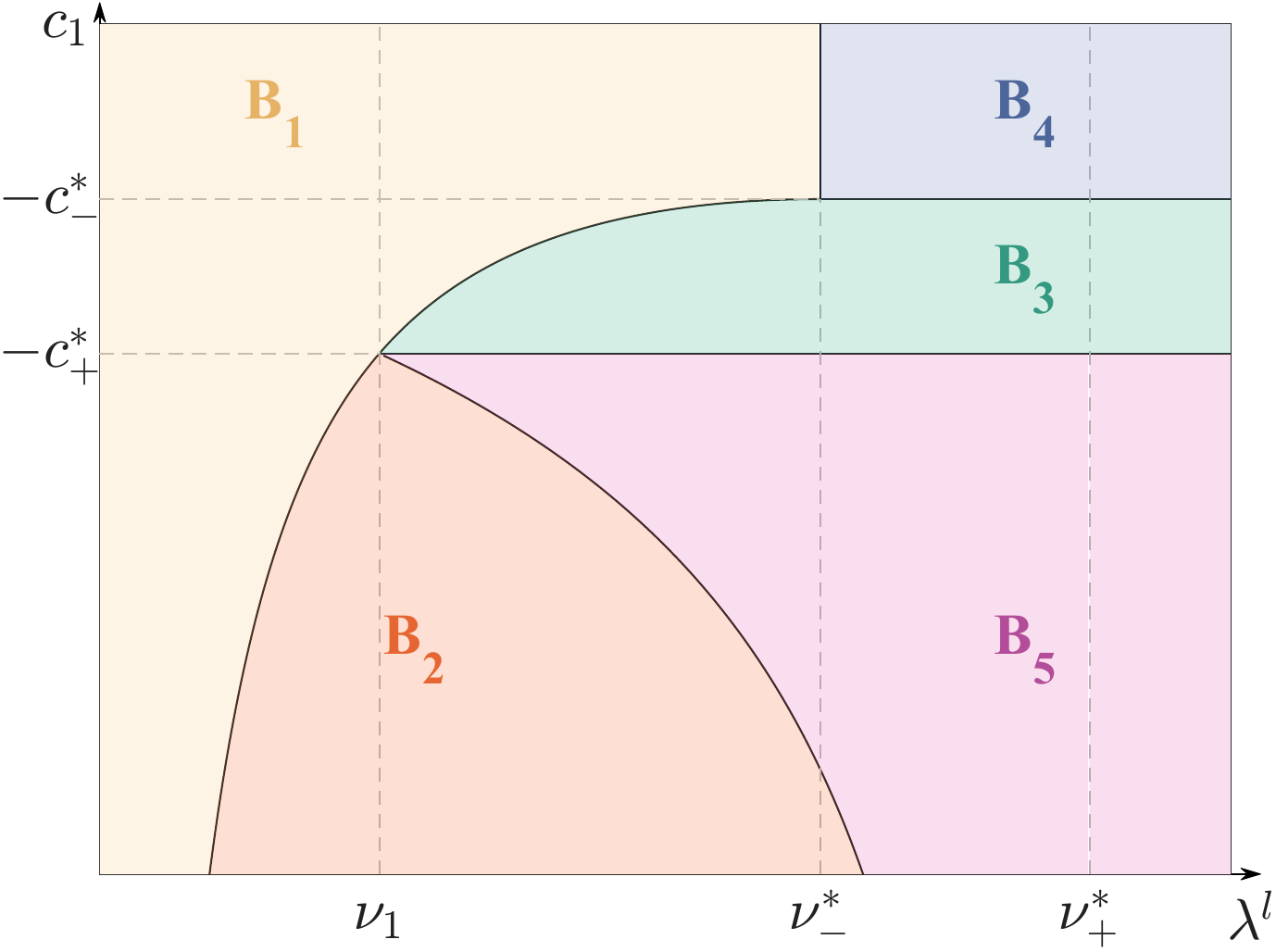}}
	\caption{Regional division of spreading speeds. }
	\label{fig:images}
	\end{figure}

\begin{theorem}\label{thm:right}
Assume that {\rm (J), (F1)--(F4)} and {\rm(\text{IC}$^\lambda)$} hold with some $\lambda\in(0,\infty)$. Then the rightward spreading speed $c_{r}$ of \eqref{eq:main} is continuous in $\{(\lambda^{r},c_{1})|~\lambda^{r}\in \mathbb{R}^{+}, c_{1}\in \mathbb{R}\}$. Moreover, it can be expressed as follows:
\begin{equation*}
    c_{r}=
\begin{cases}
    c_+(\lambda^r)  & \text{if } (\lambda^{r},c_{1})\in A_{1},\\
    c_-(\underline{p})  & \text{if }(\lambda^{r},c_{1})\in A_{2},\\
    c_-(\bar{p}) & \text{if }(\lambda^{r},c_{1})\in A_{3},\\
    c_{+}^{*} & \text{if }(\lambda^{r},c_{1})\in A_{4},\\
    c_{-}^{*}& \text{if }(\lambda^{r},c_{1})\in A_{5}.
\end{cases}
\end{equation*}
Here, for $(\lambda^r,c_1)\in A_2$, $\underline{p}(\lambda^r,c_1)$ is the smallest root of
\[
c_1p-\mathbf{H}_-(p)=c_1\lambda^r-\mathbf{H}_+(\lambda^r),
\]
and for $c_1\in[c_+^*,\bar{c}_1)$, $\bar{p}(c_1)$ is the smallest root of
\[
c_1p-\mathbf{H}_-(p)=\mathbf{L}_+(c_1),
\]
with $\bar{c}_1$ the unique positive number such that $\bar{p}(\bar{c}_1)=\nu_-^*$.
Besides, $A_{i}$, $i=1,\cdots,5$ are given by (see also Figure \ref{fig:images}-(a))
\begin{eqnarray*}
&&A_{1}: \left\{\,(\lambda^{r}, c_1)\,\left| ~\lambda^r\in (0,\nu^{*}_{+}) \text{ and } \,c_{1}\leq c_+(\lambda^r)\right.\right\},\\
&&A_{2}:
\left\{(\lambda^{r}, c_1)
\left|
\begin{array}{l}
\lambda^{r}\in (0,\nu^{*}_{-}] \text{ and } \,c_{1}>c_+(\lambda^r), ~~\text{or}\\
\lambda^{r}\in(\nu^{*}_{-},\nu^{*}_{+}] \text{ and } \,c_+(\lambda^r)<c_{1}<g(\lambda^{r}),~~\text{or}\\
\lambda^r\in (\nu^{*}_{+},\Psi_{+}(\bar c_{1})) \text{ and }\, \mathbf{H}_{+}'(\lambda^{r})\le c_{1}<g(\lambda^{r})
\end{array}
\right.
\right\},\\
&&A_{3}:
\left\{(\lambda^{r}, c_1)\,
\left|
\begin{array}{l}
\lambda^{r}\in(\nu^{*}_{+},\Psi_{+}(\bar c_{1})) \text{ and }\,c_+^*\le c_1<  \mathbf{H}_{+}'(\lambda^{r}),~~\text{or}\\
\lambda^{r}\in[\Psi_{+}(\bar c_{1}),+\infty) \text{ and }\,c_+^*\le c_1<\bar c_{1}
\end{array}
\right.
\right\}, \\
&&A_{4}: \left\{\,(\lambda^{r}, c_1)\,\left| ~\lambda^{r}\in [\nu^{*}_{+},+\infty) \text{ and }\,c_{1}< c^{*}_{+} \right. \right\},\\
&&A_{5}:
\left\{(\lambda^{r}, c_1)\,
\left|
\begin{array}{l}
\lambda^r\in (\nu^{*}_{-},\Psi_{+}(\bar c_{1})) \text{ and }\,c_{1}\geq g(\lambda^{r}), ~~\text{or}\\
\lambda^{r}\in [\Psi_{+}(\bar c_{1}),+\infty)\text{ and }\,c_{1}\geq \bar{c}_{1}
\end{array}
\right.
\right\},
\end{eqnarray*}
where $g(\cdot)$ is given by \eqref{eq:g}.
\end{theorem}

\begin{theorem}\label{thm:left}
Assume that {\rm (J), (F1)--(F4)} and {\rm(\text{IC}$^\lambda)$} hold with some $\lambda\in(0,\infty)$. Then the leftward spreading speed $c_{l}$ of \eqref{eq:main} is continuous in $\{(\lambda^{l},c_{1})|~\lambda^{l}\in \mathbb{R}^{+}, c_{1}\in \mathbb{R}\}$. Moreover, it can be expressed as follows:
\begin{equation*}
    c_{l}=
\begin{cases}
    c_-(\lambda^l)  & \text{if } (\lambda^{l},c_{1})\in B_{1},\\
    c_-(p_*)  & \text{if } (\lambda^{l},c_{1})\in B_{2},\\
    -c_1, & \text{if } (\lambda^{l},c_{1})\in B_{3},\\
    c_-^*, & \text{if } (\lambda^{l},c_{1})\in B_{4},\\
    c_+^*, & \text{if } (\lambda^{l},c_{1})\in B_{5}.
\end{cases}
\end{equation*}
Here, for $(\lambda^l,c_1)\in B_2$, $p_*:=p_*(\lambda^l,c_1)$ is the smallest root of
\[
-c_1p-\mathbf{H}_+(p)=-c_1\lambda^l-\mathbf{H}_-(\lambda^l).
\]
Besides, $B_{i}$, $i=1,\cdots,5$ are given by (see also Figure \ref{fig:images}-(b))
\begin{eqnarray*}
&&B_{1}: \left\{(\lambda^{l}, c_1)\,\left| ~\lambda^l\in (0,\nu^{*}_{-}] \text{ and } \,c_{1}> -c_-(\lambda^l)\right.\right\},\\
&&B_{2}:
\left\{(\lambda^{l}, c_1)
\left|
\begin{array}{l}
\lambda^{l}\in (0,\nu_1) \text{ and } \,c_{1}\le-c_-(\lambda^l), ~~\text{or}\\
\lambda^{l}\in[\nu_1,\nu^{*}_{+}) \text{ and } \,c_{1}\le-k(\lambda^{l})
\end{array}
\right.
\right\},\\
&&B_{3}:
\left\{(\lambda^{l}, c_1)
\left|
\begin{array}{l}
\lambda^{l}\in[\nu_1,\nu_-^*) \text{ and }\,-c_+^*\le c_1\le-c_-(\lambda^l),~~\text{or}\\
\lambda^{l}\in[\nu_-^*,+\infty) \text{ and }\,-c_+^*\leq c_1\le-c_-^*
\end{array}
\right.
\right\}, \\
&&B_{4}: \left\{(\lambda^{l}, c_1)\,\left| ~\lambda^{l}\in [\nu^{*}_{-},+\infty) \text{ and }\,c_{1}> -c^{*}_{-} \right. \right\},\\
&&B_{5}:
\left\{(\lambda^{l}, c_1)
\left|
\begin{array}{l}
\lambda^l\in (\nu_1,\nu_+^*) \text{ and }\,-k(\lambda^{l})< c_{1}<-c_+^*, ~~\text{or}\\
\lambda^{l}\in [\nu_+^*,+\infty) \text{ and }\,c_{1}<-c_+^*
\end{array}
\right.
\right\},
\end{eqnarray*}
where $\nu_1>0$ is the smallest root of $c_+^*=\frac{\mathbf{H}_-(\nu)}{\nu}$ and $k(\cdot)$ is given by \eqref{eq:k}.
\end{theorem}

The next theorem establishes the explicit formulas for the spreading speeds when  (\text{IC}$^\infty)$ holds.

\begin{theorem}\label{thm:vis-infty}
Assume that {\rm (J), (F1)--(F4)} and {\rm(\text{IC}$^\infty)$} hold.  Then the rightward (resp. leftward) spreading speed $c_{r}$ (resp. $c_{l}$) is given explicitly by
	\begin{eqnarray*}
&c_{r}=\left\{\begin{aligned}
&c_{-}^{*},&& c_1\in[\bar{c}_1,+\infty),\\
&c_-(\bar{p}(c_1)), &&c_1\in(c_+^*,\bar{c}_1),\\
&c_+^*,&& c_1\in(-\infty, c_+^*],
\end{aligned}
\right. \quad
&c_l=\left\{\begin{aligned}
&c_-^*,&&c_1\in[ -c_-^*,+\infty),\\
&-c_1,&& c_1\in(-c_+^*, -c_-^*),\\
&c_+^*,&& c_1\in(-\infty, -c_+^*].
\end{aligned}
\right.
\end{eqnarray*}
\end{theorem}
\begin{remark}\label{remark:local-refs}
\rm{
The local versions of Theorems \ref{thm:right}, \ref{thm:vis-infty} are proved in \cite{Lam2022}, while Theorem \ref{thm:vis-infty}, for the cases when $c_{r}=c_{+}^{*}$ and $c_{l}=c_{-}^{*}$ where the intrinsic speed was selected, is a consequence of \cite{Yi2025}.
}
\end{remark}

To quantify the effect of time delay $\tau$ and the associated locking phenomenon, we define the critical delay thresholds for given $c_1 \in (-c_+^*(0), 0)$ and $\lambda^l > 0$:
\begin{equation*}
\begin{aligned}
&\tau_+ := \inf\,\{\tau\ge 0 : c_+^*(\tau) \le -c_1\}, \quad &&\tau_- := \inf\,\{\tau\ge 0 : c_-^*(\tau) \le -c_1\},\\
&\tau_\lambda := \inf\,\{\tau\ge 0 : c_-(\lambda^l;\tau) \le -c_1\}, \quad &&\hat\tau := \inf\,\{\tau\ge0 : \nu_-^*(\tau) \le \lambda^l\},
\end{aligned}
\end{equation*}
with the standard notation that $\inf \emptyset = \infty$. The rigorous well-posedness of these thresholds and their ordering properties will be established in Section \ref{Sec:tau}.
Then the following holds.

\begin{theorem}\label{thm:tau}
Assume that {\rm (J), (F1)--(F4)} and {\rm(\text{IC}$^\lambda$)} hold for some $\lambda\in(0,\infty]$.
The rightward spreading speed \(c_r(\lambda^r,c_1;\tau)\) is strictly decreasing in \(\tau\ge 0\). The leftward spreading speed \(c_l(\lambda^l,c_1;\tau)\) is non-increasing in \(\tau\ge 0\), and exhibits a locking regime.
Specifically, for a given $\lambda^{l}>0$, the leftward spreading speed $c_l(\lambda^l,c_1;\tau)$ remains locally constant with respect to $\tau$ (i.e., habitat locking occurs) if and only if $c_{1}\in(-c^{*}_{+}(0),0)$ and $\tau$ belongs to the interval $\mathcal{I}_{\rm lock}$, where
\[
\mathcal{I}_{\rm lock} =
\begin{cases}
[\tau_\lambda,\tau_+], & \text{if } \hat{\tau}\ge\tau_\lambda,\\
[\tau_-,\tau_+], & \text{if } \hat{\tau}<\tau_\lambda.
\end{cases}
\]
Outside this interval, $c_l(\lambda^l,c_1;\tau)$ is strictly decreasing.
\end{theorem}

\begin{figure}[htbp]
  \centering
  \includegraphics[width=1\textwidth]{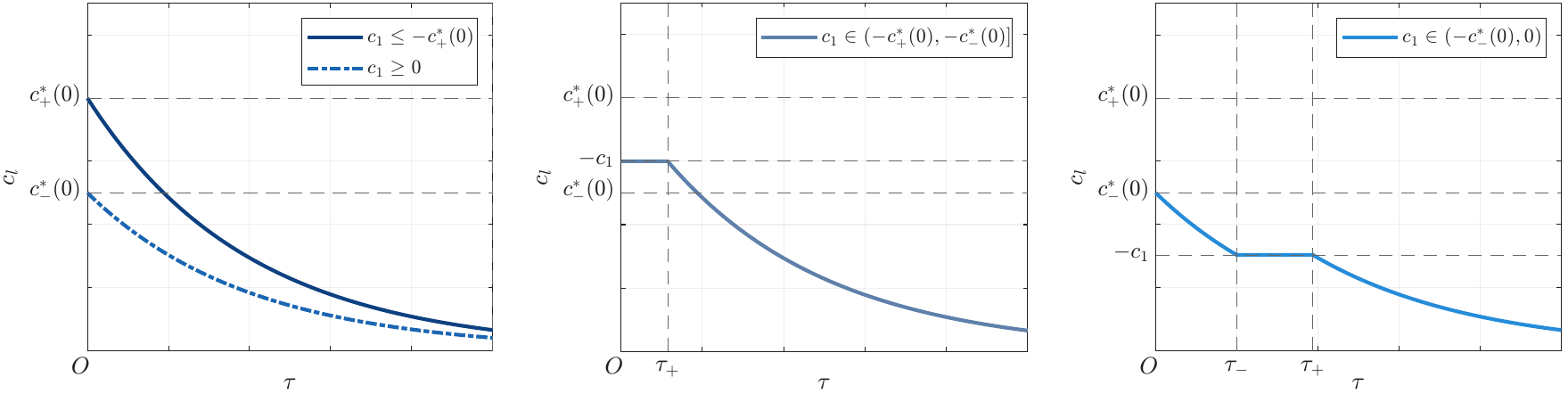}
  \caption{Profile of the leftward speed \(c_l\) versus time delay \(\tau\).}
  \label{fig:images1}
\end{figure}

\begin{remark}\label{remark:local-new}
\rm{

(i) \textbf{Laplacian diffusion and completeness:} All classification regimes and delay effects established in Theorems \ref{thm:right}--\ref{thm:vis-infty} and \ref{thm:tau} apply consistently to the local model \eqref{eq:local}, differing only in the explicit form of the limiting Hamiltonian (i.e., replacing the integral operator $\int_{\mathbb{R}}J(y)e^{py}dy-1$ with $p^{2}$ in \eqref{gamma}). While the analytical framework for the local case remains largely analogous to the nonlocal one, we have provided a rigorous proof for Lemma \ref{lem:w^vareps-bound-2} (the most distinctive component in the Laplacian setting) in Appendix \ref{App:Local}. Hence our formulas also give a complete description for the local case, extending the results of \cite{Lam2022,Yi2025} (see \ Remark~\ref{remark:local-refs}).

(ii) \textbf{Delay-insensitivity and locking threshold:} Biologically, Theorem \ref{thm:tau} highlights a clear deviation from classical delay models. While maturation delay strictly decelerates propagation speed in homogeneous settings, this monotonic reduction is disrupted in shifting habitats once the locking phenomenon occurs. Specifically, there exists a critical threshold for the delay parameter that delineates the locking regime; within this regime, the spreading speed is forced to match the habitat shifting speed $c_1$,  thereby neutralizing any further impact of the time delay. We refer to Figure \ref{fig:images1} for an illustration of this threshold behavior under condition (IC)$^{\infty}$.
}
\end{remark}

\section{Hamilton-Jacobi limits}\label{Sec:HJ}
\noindent

We introduce the hyperbolic scaling and the WKB transformation as follows:
\begin{equation}\label{w_u_varepsilon}
u^\varepsilon(t,x) = u\left( \frac{t}{\varepsilon}, \frac{x}{\varepsilon} \right), \quad
w^\varepsilon(t,x) = -\varepsilon \ln u^\varepsilon(t,x).
\end{equation}
A simple calculation gives that $w^\varepsilon$ satisfies that
\begin{align}\label{w^varepsilon}
d\int_{\mathbb{R}}J(y)\left[e^{-\frac{w^\varepsilon(t,x-\varepsilon y)-w^\varepsilon(t,x)}{\varepsilon}}-1\right]dy + \partial_tw^\varepsilon -\mu  + \mu \frac{f^\varepsilon(x-c_1t,u^\varepsilon(t-\varepsilon\tau,x))}{u^\varepsilon(t,x)} =0,
\end{align}
for all $t>0$ and $x\in\mathbb{R}$, where $f^\varepsilon(z,u)=f\left(\frac{z}{\varepsilon},u\right)$.

\begin{lemma}\label{lem:w^vareps-bound-1}
Assume that {\rm (J)} and {\rm (F1)--(F4)} hold. Let $w^\varepsilon$ be a solution of \eqref{w^varepsilon} with initial condition {\rm(\text{IC}$^\lambda)$.} Then for any $\delta>0$, there exists a positive constant $K$ independent of small $\varepsilon$, such that
\begin{align*}
\max\left\{(\lambda^r-\delta)x_++(\lambda^l-\delta)x_--K(t+\varepsilon), -\varepsilon
\ln{u^* }\right\}\leq w^\varepsilon(t,x)
\end{align*}
and
\begin{align*}
w^\varepsilon(t,x)\leq(\lambda^r+\delta)x_++(\lambda^l+\delta)x_-+K(t+\varepsilon)
\end{align*}
for all $(t,x)\in[-\varepsilon\tau,+\infty)\times\mathbb{R}$, where $x_+=\max\{x,0\}$ and $x_-=\max\{-x,0\}$.
\end{lemma}

\begin{proof}
By assumption (F1), it follows that $u^\varepsilon$ is a supersolution to 
\[
\varepsilon u_t-d\int_{\mathbb{R}}J(y)\left[u(t,x-\varepsilon y)-u(t,x)\right]dy +\mu u\ge0 \quad\text{ in } (0,\infty)\times\mathbb{R}.
\]
Noting that $w^\varepsilon = -\varepsilon\ln u^\varepsilon$, we have
\begin{align}\label{eq:w-sub}
\partial_tw^\varepsilon+d\int_{\mathbb{R}}J(y)\left[e^{-\frac{w^{\varepsilon}(t,x-\varepsilon y)-w^\varepsilon(t,x)}{\varepsilon}}-1\right]dy-\mu \le0 &\quad\text{ in } (0,\infty)\times\mathbb{R}.
\end{align}
In view of initial condition $(\text{IC}^\lambda)$, for any small $\delta\in(0,\min\{\lambda^r,\lambda^l\})$, there exist constants $b>a>0$ such that
\[
ae^{-[(\lambda^r+\delta)x_++(\lambda^l+\delta)x_-]}\leq \phi(\theta,x)\leq be^{-[(\lambda^r-\delta)x_++(\lambda^l-\delta)x_-]} \quad \text{ for all } (\theta,x)\in[-\varepsilon\tau,0]\times\mathbb{R}.
\]
Via WKB transformation \eqref{w_u_varepsilon}, this implies
\[
(\lambda^r-\delta)x_++(\lambda^l-\delta)x_--\varepsilon\ln b\leq w^\varepsilon(\theta,x)\leq(\lambda^r+\delta)x_++(\lambda^l+\delta)x_--\varepsilon\ln a
\]
for all $(\theta,x)\in[-\varepsilon\tau,0]\times\mathbb{R}$.

Define
\[
\overline{w}^\varepsilon(t,x)=(\lambda^r+\delta)x_++(\lambda^l+\delta)x_-+K_1(t+\varepsilon),
\]
where $K_1>0$ is chosen sufficiently large to satisfy $K_1\geq\max\left\{|\ln a|, d+\mu\right\}$.
By construction, the initial condition $w^\varepsilon(\theta,x)\leq\overline{w}^\varepsilon(\theta,x)$ holds for all $(\theta,x)\in[-\varepsilon\tau,0]\times\mathbb{R}$. A direct computation implies that $\overline{w}^\varepsilon$ is a supersolution to \eqref{eq:w-sub}, so the comparison principle yields
\[
w^\varepsilon(t,x)\leq\overline{w}^\varepsilon(t,x) \quad\text{for all } (t,x)\in[-\varepsilon\tau,+\infty)\times\mathbb{R}.
\]

Define
\[
\underline{w}_1^\varepsilon(t,x)= (\lambda^r-\delta)x-K_2(t+\varepsilon) \quad\text{ and } \quad\underline{w}_2^\varepsilon(t,x)= -(\lambda^l-\delta)x-K_2(t+\varepsilon),
\]
where $K_2>0$ is sufficiently large to satisfy
\[
K_2\geq\max\left\{d\int_{\mathbb{R}}J(y)e^{(\max\{\lambda^r,\lambda^l\}-\delta)y}dy-d+\mu K_3,|\ln b|\right\},
\]
with \( K_3 = \sup\limits_{z \in \mathbb{R}} |\partial_u f(z,0)| \). By assumption (F1) and the WKB transformation \eqref{w_u_varepsilon}, we have
\begin{align*}
&d\int_{\mathbb{R}} J(y) \left[ e^{-\frac{\underline{w}_1^\varepsilon(t,x - \varepsilon y) - \underline{w}_1^\varepsilon(t,x)}{\varepsilon}} - 1 \right] dy +\partial_t \underline{w}_1^\varepsilon - \mu + \mu \frac{f^\varepsilon\left(x - c_1t, e^{-\frac{\underline{w}_1^\varepsilon(t - \varepsilon\tau, x)}{\varepsilon}}\right)}{e^{-\frac{\underline{w}_1^\varepsilon(t, x)}{\varepsilon}}} \\
\leq &d\int_{\mathbb{R}} J(y) \left[ e^{-\frac{\underline{w}_1^\varepsilon(t,x - \varepsilon y) - \underline{w}_1^\varepsilon(t,x)}{\varepsilon}} - 1 \right] dy +\partial_t \underline{w}_1^\varepsilon - \mu + \mu \partial_u f^\varepsilon(x - c_1t, 0) e^{\frac{\underline{w}_1^\varepsilon(t,x) - \underline{w}_1^\varepsilon(t - \varepsilon\tau, x)}{\varepsilon}} \\
=& d\int_{\mathbb{R}} J(y)\left(e^{(\lambda^r-\delta)y} - 1\right)dy - K_2 - \mu + \mu K_3 e^{-K_2\tau} \\
\leq& 0.
\end{align*}
Then $\underline{w}_1^\varepsilon$ is a subsolution of \eqref{w^varepsilon}.
For $\underline{w}_2^\varepsilon$, a parallel computation gives
\[
d\int_{\mathbb{R}} J(y)\left(e^{-(\lambda^l-\delta)y} - 1\right)dy - K_2 - \mu + \mu K_3 e^{-K_2\tau} \leq 0,
\]
and therefore, $\underline{w}_2^\varepsilon$ is also a subsolution.
Taking $K = \max\{K_1, K_2\}$ unifies the upper and lower bounds, which completes the proof.
\end{proof}

\begin{lemma}\label{lem:w^vareps-bound-2}
Assume that {\rm (J)} and {\rm (F1)--(F4)} hold. Let $w^\varepsilon$ be a solution of \eqref{w^varepsilon} with initial condition {\rm(\text{IC}$^\infty)$.} Then for any compact subset $Q_1$ of $(0,\infty)\times\mathbb{R}$, there exists a positive constant $C=C(Q_1)$ and $\varepsilon_0>0$ such that for any $\varepsilon\in(0,\varepsilon_0)$,
\[
\sup_{Q_1}|w^\varepsilon|\leq C  .
\]
\end{lemma}
\begin{proof}
Since $0\le u(t,x)\le u^*$ for all \( (t,x) \in [-\tau,\infty)\times\mathbb{R} \), we have $w^\varepsilon\ge -\varepsilon \ln u^*$ in \( [-\varepsilon\tau,\infty)\times\mathbb{R} \). Hence it suffices to establish an upper bound for $w^\varepsilon$. Observe that $f(z,u) \ge 0$ on $\mathbb{R} \times [0,\infty)$, which gives
\begin{align*}
u_t &= d[J * u - u] - \mu u + \mu f\left(x - c_1 t, u(t-\tau,x)\right) \\
&\ge d[J * u - u] - \mu u .
\end{align*}
The upper bound for $w^\varepsilon$ now follows from the argument of \cite[Theorem 4.5]{Liang2020}. We end this proof.
\end{proof}

As a consequence of Lemmas \ref{lem:w^vareps-bound-1} and \ref{lem:w^vareps-bound-2}, the half-relaxed limits
\begin{align*}
w_*(t,x):=\liminf_{\substack{\varepsilon \rightarrow 0\\(s,y)\rightarrow(t,x)}}w^{\varepsilon}(s,y) \quad \text { and } \quad w^*(t,x):=\limsup_{\substack{\varepsilon \rightarrow 0\\(s,y)\rightarrow(t,x)}} w^{\varepsilon}(s,y)
\end{align*}
are well-defined. Then, we have the following results.
\begin{lemma}\label{lem:w^*_*}
Assume that {\rm (J)} and {\rm (F1)--(F4)} hold. Then $w^*(t,0)=w_{*}(t,0)=0$ for all $t\in(0,\infty)$, and
\begin{align*}
& w_*(0,x)=w^*(0,x)=\begin{cases} \lambda^rx, &x\geq0,\\-\lambda^l x, &x<0, \end{cases}\quad \quad\text{if } {{\rm(IC}}^\lambda) \text{ holds},\\
& w^*(0,x)=w_{*}(0,x)=\infty \quad\text{ for } x\neq0, \quad\quad\text{if } {{\rm(IC}}^\infty) \text{ holds}.
\end{align*}
\end{lemma}
\begin{proof}
From the definition of $w^\varepsilon$ and the bound $0\le u\le u^*$, we have
$w^\varepsilon\ge -\varepsilon\ln u^*$, and hence
\[
w_*(t,x) \geq \liminf_{\varepsilon\to0} \left(-\varepsilon \ln u^*\right) = 0 \quad \text{for}\quad (t,x) \in (0,\infty)\times\mathbb{R}.
\]
In particular, \( w_*(t,0) \geq 0 \) holds for all \( t > 0 \). To obtain an upper bound, let $\underline{u}$ denote the unique solution of \eqref{eq:main} with $f$ replaced by $f^-$. By the comparison principle, $0\leq\underline{u}(t,x)\le u(t,x)$ on $[-\tau,\infty)\times\mathbb{R}$. Under assumptions (F2), the results of \cite{Fang2014,Yi2015} provide a spreading speed $\tilde{c}_->0$ and a number $\eta_0\in(0,\tilde{c}_-)$ such that \( \liminf\limits_{t\to\infty}\inf\limits_{|x|\leq(\tilde{c}_--\eta_0)t}\underline{u}(t,x) > 0 \).
Consequently,
\[
w^*(t,0)=\limsup_{\substack{\varepsilon \rightarrow 0\\(s,x)\rightarrow(t,0)}} w^{\varepsilon}(s,x)\leq\limsup_{\substack{\varepsilon \rightarrow 0\\(s,y)\rightarrow(t,0)}} -\varepsilon \ln \underline{u}\left(\frac{s}{\varepsilon},\frac{x}{\varepsilon}\right)\leq0 \quad\text{for all } t>0.
\]
Thus, $w^*(t,0)=w_{*}(t,0)=0$ for all $t\in(0,\infty)$.

We now prove the remaining assertions in this lemma. If (IC$^\lambda$) holds, Lemma~\ref{lem:w^vareps-bound-1} gives, after sending $\varepsilon\to0$,
\[
\max\{\lambda^rx_++\lambda^lx_--Kt, 0\}\leq w_*(t,x)\leq w^*(t,x)\leq\lambda^rx_++\lambda^lx_-+Kt
\]
for $(t,x)\in[0,+\infty)\times\mathbb{R}$. Letting $t\to0$ proves the first case. If (IC$^\infty$) holds, fix an arbitrary $\lambda>0$ and choose $C_1>0$ large enough so that $\phi(\theta,x)\le C_1e^{-\lambda |x|}$ for $(\theta,x)\in[-\tau,0]\times\mathbb{R}$. Let $\tilde u$ be the solution of \eqref{eq:main} with initial data $\min\{C_1e^{-\lambda |x|},u^*\}$. By comparison principle, $u(t,x)\le\tilde u(t,x)$ on $[-\tau,\infty)\times\mathbb{R}$.
Repeating the proof of Lemma~\ref{lem:w^vareps-bound-1}, there exists a constant $K_4>0$ such that
\[
\max\left\{\lambda |x|-K_4(t+\varepsilon),\,-\varepsilon\ln C_1\right\}\le \tilde w^\varepsilon(t,x)
\qquad\text{for } (t,x)\in[-\varepsilon\tau,\infty)\times\mathbb{R},
\]
where $\tilde w^\varepsilon:=-\varepsilon\ln\tilde{u}(t/\varepsilon,x/\varepsilon)$.
Since $u\le\tilde u$, passing to the limit $\varepsilon\to0$ gives
\[
\max\left\{\lambda |x|-K_4t,\,0\right\}\le w_*(t,x)
\qquad\text{for } (t,x)\in(0,\infty)\times\mathbb{R}.
\]
For any $x\neq0$, letting $t\to0$ in the inequality above yields $\lambda |x|\le w_*(0,x)$. Because $\lambda>0$ is arbitrary, taking $\lambda\to\infty$ shows that $w_*(0,x)=\infty$ for every $x\neq0$. Consequently, $w^*(0,x)\ge w_*(0,x)=\infty$ for all $x\neq0$, which completes the proof.
\end{proof}

\begin{lemma}\label{lem:w-sup-sol}
Assume that {\rm (J)}, {\rm (F1)} hold and {\rm\text{(IC}$^\lambda)$} hold for some $\lambda\in(0,+\infty]$. Then $w_{*}$ is a viscosity supersolution of \eqref{eq:constrained-HJ}.
\end{lemma}
\begin{proof}
By Lemma \ref{lem:w^*_*}, $w_*(t,x)\geq0$ in $(0,\infty)\times\mathbb{R}$. Fix a smooth test function $\varphi,$ without loss of generality, assume that $w_{*}-\varphi$ has a strict local minimum at some point $(t_{0},x_{0})\in(0,\infty)\times\mathbb{R}.$ It then suffices to show that $H^*(x/t,\partial_{x}\varphi,\partial_t\varphi)\geq0$ at $(t_{0},x_{0}).$

By the definition of $w_*$, there exist sufficiently small constants $r_1>0$ and $x_1>0$, and sequences $\{\varepsilon_n\}$ and $\{(t_{\varepsilon_n},x_{\varepsilon_n})\}$ with $|t_{\varepsilon_n}-t_0|\leq r_1$ and $|x_{\varepsilon_n}-x_0|\leq x_1$ such that $w^{\varepsilon_{n}}(t,x)-\varphi$ has a local minimum at $(t_{\varepsilon_n},x_{\varepsilon_n})$ over $Q_1:=(t_0-r_1,t_0+r_1)\times(x_0-2x_1,x_0+2x_1)$ with
\begin{align*}
w^{\varepsilon_{n}}(t_{\varepsilon_n},x_{\varepsilon_n})\to w_{*}(t_{0},x_{0})\quad \text{and}\quad(t_{\varepsilon_n},x_{\varepsilon_n})\to(t_{0},x_{0})\ \quad\text{as}\quad\ n\to\infty.
\end{align*}
By the definition of $(t_{\varepsilon_n},x_{\varepsilon_n})$ being the minimum, we have the inequality
\begin{align}\label{w_varep_n_2}
w^{\varepsilon_{n}}(t_{\varepsilon_n},x_{\varepsilon_n})-\varphi(t_{\varepsilon_n},x_{\varepsilon_n})\leq w^{\varepsilon_{n}}(t,x)-\varphi(t,x)~~~\ \text{for}~~~\ (t,x)\in Q_1.
\end{align}
From equation \eqref{w^varepsilon}, it follows that at $(t_{\varepsilon_n}, x_{\varepsilon_n})$,
\begin{equation}\label{phi}
\begin{aligned}
0=&d\int_{\mathbb{R}}J(y)\left[e^{-\frac{w^{\varepsilon_n}(t_{\varepsilon_n},x_{\varepsilon_n}-\varepsilon_n y)-w^{\varepsilon_n}(t_{\varepsilon_n},x_{\varepsilon_n})}{\varepsilon_n}}-1\right]dy + \partial_t\varphi -\mu  \\
&+ \mu \frac{f^{\varepsilon_n}(x_{\varepsilon_n}-c_1t_{\varepsilon_n},u^{\varepsilon_n}(t_{\varepsilon_n}-\varepsilon_n\tau,x_{\varepsilon_n}))}{u^{\varepsilon_n}(t_{\varepsilon_n},x_{\varepsilon_n})}.
\end{aligned}
\end{equation}

First, we analyze the nonlocal term. Choose a large constant $L>0$ such that $\text{supp}(J)\Subset[-L,L]$. For sufficiently large $n$, we have $|x_{\varepsilon_n} - x_0| < x_1$ and $|\varepsilon_n y| < x_1$ for $y \in [-L, L]$, which implies $x_{\varepsilon_n} - \varepsilon_n y \in (x_0 - 2x_1, x_0 + 2x_1)$. Using \eqref{w_varep_n_2}, we obtain
\begin{align*}
d\int_{\mathbb{R}}J(y)e^{-\frac{w^{\varepsilon_n}(t_{\varepsilon_n},x_{\varepsilon_n}-\varepsilon_n y)-w^{\varepsilon_n}(t_{\varepsilon_n},x_{\varepsilon_n})}{\varepsilon_n}}dy\leq\int_{-L}^{L}J(y)e^{-\frac{\varphi(t_{\varepsilon_n},x_{\varepsilon_n}-\varepsilon_n y)-\varphi(t_{\varepsilon_n},x_{\varepsilon_n})}{\varepsilon_n}}dy:=J_1^{L,\varepsilon_n}
\end{align*}
By the smoothness of $\varphi$, we have
\[
\lim_{n\to\infty}-\frac{\varphi(t_{\varepsilon_n},x_{\varepsilon_n}-\varepsilon_n y)-\varphi(t_{\varepsilon_n},x_{\varepsilon_n})}{\varepsilon_n}=y\partial_x\varphi(t_0,x_0) \quad\text{ uniformly in } [-L,L].
\]
The dominated convergence theorem then implies
\[
\lim_{n\to\infty}J_1^{L,\varepsilon_n}=\int_{-L}^{L}J(y)e^{y\partial_x\varphi(t_0,x_0)}dy.
\]

We now analyze the reaction term. Under condition {\rm{(F1)}}, along with \eqref{w_u_varepsilon} and \eqref{w_varep_n_2}, we obtain the following estimate at $(t_{\varepsilon_n}, x_{\varepsilon_n})$,
\begin{align*}
\frac{f^{\varepsilon_n}(x_{\varepsilon_n}-c_1t_{\varepsilon_n},u^{\varepsilon_n}(t_{\varepsilon_n}-\varepsilon_n\tau,x_{\varepsilon_n}))}{u^{\varepsilon_n}(t_{\varepsilon_n},x_{\varepsilon_n})} \leq& \partial_uf^{\varepsilon_n}(x_{\varepsilon_n}-c_1t_{\varepsilon_n},0)\frac{u^{\varepsilon_n}(t_{\varepsilon_n}-\varepsilon_n\tau,x_{\varepsilon_n})}{u^{\varepsilon_n}(t_{\varepsilon_n},x_{\varepsilon_n})} \\
=&\partial_uf^{\varepsilon_n}(x_{\varepsilon_n}-c_1t_{\varepsilon_n},0)e^{\frac{w^{\varepsilon_n}(t_{\varepsilon_n},x_{\varepsilon_n})-w^{\varepsilon_n}(t_{\varepsilon_n}-\varepsilon_n\tau,x_{\varepsilon_n})}{\varepsilon_n}}\\
\leq&\partial_uf^{\varepsilon_n}(x_{\varepsilon_n}-c_1t_{\varepsilon_n},0)e^{\frac{\varphi(t_{\varepsilon_n},x_{\varepsilon_n})-\varphi(t_{\varepsilon_n}-\varepsilon_n\tau,x_{\varepsilon_n})}{\varepsilon_n}}.
\end{align*}
Again, the smoothness of $\varphi$ yields that
\[
\lim_{n\to\infty}\frac{\varphi(t_{\varepsilon_n},x_{\varepsilon_n})-\varphi(t_{\varepsilon_n}-\varepsilon_n\tau,x_{\varepsilon_n})}{\varepsilon_n}=\tau\partial_t\varphi(t_0,x_0).
\]

Taking $n\to\infty$ in \eqref{phi} and using the above estimates, we obtain
\begin{align*}
0\leq&d\int_{\mathbb{R}}J(y)\left(e^{y\partial_x\varphi(t_0,x_0)}-1\right)dy +\partial_t\varphi(t_0,x_0)-\mu+\mu R(x_0/t_0)e^{\tau\partial_t\varphi(t_0,x_0)},
\end{align*}
which completes the proof.
\end{proof}

\begin{lemma}\label{lem:rho-sup-sol}
Assume that {\rm(IC}$^\lambda)$ holds for some $\lambda\in(0,+\infty]$. Then $\rho_{*}(s) := w_{*}(1,s)$ is a viscosity supersolution of the following Hamilton-Jacobi equation,
\begin{align}\label{rho1}
\min\left\{ \rho,H(s,\rho',\rho-s\rho') \right\} = 0 \quad \text{ in } s\in(0,\infty) \text{ or } s\in(-\infty,0) .
\end{align}
Particularly, the following boundary conditions hold in the classical sense,
\begin{align}\label{rho1-boundary}
\rho_{*}(0) = 0, \quad \liminf_{s \to +\infty} \frac{\rho_{*}(s)}{s} \geq \lambda^r ~~~\text{and } \liminf_{s \to -\infty} \frac{\rho_{*}(s)}{|s|} \geq \lambda^l.
\end{align}
\end{lemma}
\begin{proof}
Following from the definition of the half-relaxed limit and WKB transformation, for any fixed $t >0$, we have
\[
w_{*}(t,x) = \liminf_{\substack{\varepsilon \to 0 \\ (s,y) \to (t,x)}} -\varepsilon \ln u\left( \frac{s}{\varepsilon}, \frac{y}{\varepsilon} \right) = t \cdot \liminf_{\substack{\varepsilon \to 0 \\ (s,y) \to (1,x/t)}} \frac{-\varepsilon}{t} \ln u\left(\frac{s}{\varepsilon/t}, \frac{y}{\varepsilon/t}\right)= t w_{*}\left( 1,\frac{x}{t}\right).
\]
Note that
\[
\rho_*(s) = w_*(1, s) = s w_*\left(\frac{1}{s}, 1\right).
\]
By Lemma \ref{lem:w^*_*} and the invariance of viscosity solution with respect to smooth transformation (see \cite[Lemma 2.3]{Liu2021-2}), $\rho_*$ is a viscosity supersolution of \eqref{rho1}. The lower semi-continuity of $w_*$ implies that $\rho_*$ is likewise lower semi-continuous and, moreover, satisfies
\[
\begin{cases}
\rho_*(0) = w_*(1, 0) = 0,\\[4pt]
\displaystyle\liminf_{s\to\infty} \frac{\rho_*(s)}{s} = \liminf_{s\to\infty} w_*\left(\frac{1}{s}, 1\right) \ge w_*(0, 1) = \lambda^r \in (0, +\infty],\\[4pt]
\displaystyle\liminf_{s\to-\infty} \frac{\rho_*(s)}{|s|} = \liminf_{s\to-\infty} w_*\left(\frac{1}{-s}, 1\right) \ge w_*(0, -1) = \lambda^l \in (0, +\infty].
\end{cases}
\]
This completes the proof.
\end{proof}

\begin{lemma}\label{lem:u^vareps=0}
Assume that {\rm(\text{IC}$^\lambda)$} holds for some $\lambda\in(0,+\infty]$. Then
\[
\limsup_{\varepsilon\to0}u^\varepsilon=0 \text{ locally uniformly in } \mathcal{A}=\{(t,x)\in(0,\infty)\times\mathbb{R}:w_{*}(t,x)>0\}.
\]
\end{lemma}
\begin{proof}
For any compact subset $Q_2 \subset \mathcal{A}$, there exists $k>0$ such that $w_* > k$ on $Q_2$. By the definition of the half-relaxed limit, we have $w^\varepsilon \ge k/2$ on $Q_2$ for all sufficiently small $\varepsilon > 0$. The WKB transformation then implies
\[
u^\varepsilon(t,x) = e^{-\frac{w^\varepsilon(t,x)}{\varepsilon}} \le e^{-\frac{k}{2\varepsilon}} \quad \text{in } Q_2.
\]
Passing to the limit $\varepsilon \to 0$ yields the desired result.
\end{proof}

\begin{lemma}\label{lem:w-sub-sol}
Assume that {\rm (J)}, {\rm (F1)} hold and {\rm (\text{IC}$^\lambda)$} holds for some $\lambda\in(0,+\infty]$. Then $w^*$ is the viscosity subsolution of \eqref{eq:constrained-HJ}.
\end{lemma}
\begin{proof}
Let $\varphi$ be a smooth test function such that $w^*-\varphi$ attains its strict local maximum at $(t_0,x_0)$ with $w^*(t_0,x_0)>0$. By the definition of $w^*$, there exist sufficiently small constants $r_2>0$ and $x_2>0$, and sequences $\{\varepsilon_n\}$ and $\{(t_{\varepsilon_n},x_{\varepsilon_n})\}$ with $|t_{\varepsilon_n}-t_0|\leq r_2$ and $|x_{\varepsilon_n}-x_0|\leq x_2$ such that $w^{\varepsilon_{n}}(t,x)-\varphi$ has a local maximum at $(t_{\varepsilon_n},x_{\varepsilon_n})$ over $Q_3:=(t_0-r_2,t_0+r_2)\times(x_0-2x_2,x_0+2x_2)$ with
\begin{align}\label{w_varepsilon_n_1}
w^{\varepsilon_n}(t_{\varepsilon_n},x_{\varepsilon_n})\to w^*(t_0,x_0)>0 \quad \text{ and } \quad(t_{\varepsilon_n},x_{\varepsilon_n})\to(t_0,x_0) ~~~\text{ as }~~~ n\to+\infty.
\end{align}
By the definition of $(t_{\varepsilon_n},x_{\varepsilon_n})$ as the maximum point in $Q_3$, we have
\begin{align}\label{w_varepsilon_n_2}
w^{\varepsilon_{n}}(t_{\varepsilon_n},x_{\varepsilon_n})-\varphi(t_{\varepsilon_n},x_{\varepsilon_n})\geq w^{\varepsilon_{n}}(t,x)-\varphi(t,x)\quad \text{for}\ (t,x)\in Q_3.
\end{align}
Then, following from \eqref{w^varepsilon}, we have at $(t_{\varepsilon_n},x_{\varepsilon_n})$,
\begin{equation}\label{w^*}
\begin{aligned}
0=&d\int_{\mathbb{R}}J(y)\left[e^{-\frac{w^{\varepsilon_n}(t_{\varepsilon_n},x_{\varepsilon_n}-\varepsilon y)-w^{\varepsilon_n}(t_{\varepsilon_n},x_{\varepsilon_n})}{\varepsilon}}-1\right]dy + \partial_t\varphi -\mu \\ &+\mu\frac{f^{\varepsilon_n}(x_{\varepsilon_n}-c_1t_{\varepsilon_n},u^{\varepsilon_n}(t_{\varepsilon_n}-\varepsilon\tau,x_{\varepsilon_n}))}{u^{\varepsilon_n}(t_{\varepsilon_n},x_{\varepsilon_n})}.
\end{aligned}
\end{equation}

First, choose a large constant $L>0$ such that $\text{supp}(J)\Subset[-L,L]$. Note that $|x_{\varepsilon_n}-x_0|<x_2$ and $|\varepsilon_ny|<x_2$ for $y\in[-L,L]$ and sufficiently large $n$. Then $x_{\varepsilon_n}-\varepsilon_ny\in(x_0-2x_2,x_0+2x_2)$. By \eqref{w_varepsilon_n_2}, we have
\begin{align*}
d\int_{\mathbb{R}}J(y)e^{-\frac{w^{\varepsilon_n}(t_{\varepsilon_n},x_{\varepsilon_n}-\varepsilon_n y)-w^{\varepsilon_n}(t_{\varepsilon_n},x_{\varepsilon_n})}{\varepsilon_n}}dy\geq\int_{-L}^{L}J(y)e^{-\frac{\varphi(t_{\varepsilon_n},x_{\varepsilon_n}-\varepsilon_n y)-\varphi(t_{\varepsilon_n},x_{\varepsilon_n})}{\varepsilon_n}}dy:=J_2^{L,\varepsilon_n}
\end{align*}
Then, the smoothness of $\varphi$ yields that
\[
\lim_{n\to\infty}-\frac{\varphi(t_{\varepsilon_n},x_{\varepsilon_n}-\varepsilon_n y)-\varphi(t_{\varepsilon_n},x_{\varepsilon_n})}{\varepsilon_n}=y\partial_x\varphi(t_0,x_0) \quad\text{ uniformly in } [-L,L].
\]
By the dominated convergence theorem, we obtain
\[
\lim_{n\to\infty}J_2^{L,\varepsilon_n}=\int_{-L}^{L}J(y)e^{y\partial_x\varphi(t_0,x_0)}dy.
\]

Next, we claim that
\begin{align}\label{f_varepsilon}
\liminf_{n\to+\infty}\frac{f^{\varepsilon_n}(x_{\varepsilon_n}-c_1t_{\varepsilon_n},u^{\varepsilon_n}(t_{\varepsilon_n}-\varepsilon_n\tau,x_{\varepsilon_n}))}{u^{\varepsilon_n}(t_{\varepsilon_n},x_{\varepsilon_n})}\geq \underline{R}(x_0/t_0)e^{\tau\partial_t\varphi(t_0,x_0)}.
\end{align}
By passing to a subsequence, we may divide the analysis into two cases:
\[
\text{(i)}\quad u^{\varepsilon_n}(t_{\varepsilon_n} - \varepsilon_n \tau, x_{\varepsilon_n}) \to 0 \quad\text{ as }  n \to \infty \quad \text{ or }\quad\text{ (ii) }\quad \inf_n u^{\varepsilon_n}(t_{\varepsilon_n} - \varepsilon_n \tau, x_{\varepsilon_n} )> 0.
\]
In case (i), we use (F1) and \eqref{w_varepsilon_n_2} to obtain
\begin{align*}
&\frac{f^{\varepsilon_n}(x_{\varepsilon_n}-c_1t_{\varepsilon_n}, u^{\varepsilon_n}(t_{\varepsilon_n} - \varepsilon_n \tau, x_{\varepsilon_n}))}{u^{\varepsilon_n}(t_{\varepsilon_n}, x_{\varepsilon_n})}\\
\geq& \left( \partial_u f^{\varepsilon_n}(x_{\varepsilon_n}-c_1t_{\varepsilon_n}, 0) - \eta_* \right) \exp\left( \frac{w^{\varepsilon_n}(t_{\varepsilon_n}, x_{\varepsilon_n}) - w^{\varepsilon_n}(t_{\varepsilon_n} - \varepsilon_n \tau, x_{\varepsilon_n})}{\varepsilon_n} \right)\\
\geq& \left( \partial_u f^{\varepsilon_n}(x_{\varepsilon_n}-c_1t_{\varepsilon_n}, 0) - \eta_* \right) \exp\left( \frac{\varphi(t_{\varepsilon_n}, x_{\varepsilon_n}) - \varphi(t_{\varepsilon_n} - \varepsilon_n \tau, x_{\varepsilon_n})}{\varepsilon_n} \right).
\end{align*}
Taking the limit as $n \to \infty$ yields
\[
\liminf_{n \to \infty} \frac{f^{\varepsilon_n}(x_{\varepsilon_n}-c_1t_{\varepsilon_n}, u^{\varepsilon_n}(t_{\varepsilon_n} - \varepsilon_n \tau, x_{\varepsilon_n}))}{u^{\varepsilon_n}(t_{\varepsilon_n}, x_{\varepsilon_n})} \geq \left( \underline{R}(x_0/t_0) - \eta_* \right) e^{\tau\partial_t \varphi(t_0, x_0)}.
\]
Since \( \eta_* > 0 \) is arbitrarily small, we obtain \eqref{f_varepsilon}.
In case (ii), observe that
\[
f^{\varepsilon_n}\!\left(x_{\varepsilon_n}-c_1t_{\varepsilon_n},\,
u^{\varepsilon_n}\big(t_{\varepsilon_n} - \varepsilon_n \tau,\, x_{\varepsilon_n}\big)\right)
\]
is bounded from below by a positive number. Meanwhile, we have
\[
u^{\varepsilon_n}(t_{\varepsilon_n}, x_{\varepsilon_n})
= \exp\big(-w^{\varepsilon_n}(t_{\varepsilon_n}, x_{\varepsilon_n})/\varepsilon_n\big) \to 0^+,
\]
where we used \eqref{w_varepsilon_n_1} and the fact that \( w^*(t_0, x_0) > 0 \).
Hence, \eqref{f_varepsilon} holds trivially.

Finally, letting $n\to+\infty$, we get from \eqref{w^*} that
\begin{align*}
0\geq d\int_{\mathbb{R}}J(y)(e^{y\partial_x\varphi(t_0,x_0)}-1)dy+\partial_t\varphi(t_0,x_0)-\mu+\mu \underline{R}(x_0/t_0)e^{\tau\partial_t\varphi(t_0,x_0)}.
\end{align*}
This ends the proof.
\end{proof}

\begin{lemma}\label{lem:rho-sub-sol}
Assume that {\rm(\text{IC}$^\lambda)$} holds for some $\lambda\in(0,+\infty]$. Then $\rho^*(s) := w^*(1,s)$ is a viscosity subsolution of \eqref{rho1}.
Particularly, the following boundary condition holds in the classical sense,
\begin{align*}
\rho^{*}(0) = 0, \quad \lim_{s \to +\infty} \frac{\rho^{*}(s)}{s} \leq \lambda^r ~~~\text{and } \lim_{s \to -\infty} \frac{\rho^{*}(s)}{|s|} \leq \lambda^l.
\end{align*}
\end{lemma}
\begin{proof}
Following the same argument as in Lemma \ref{lem:rho-sup-sol}, one can verify that \(\rho^*(s) := w^*(1,s)\) is a viscosity subsolution of \eqref{rho1}, with \(\rho^{*}(0) = 0\) and 
\[\limsup_{s \to +\infty} \frac{\rho^{*}(s)}{s} \leq \lambda^r, \quad \limsup_{s \to -\infty} \frac{\rho^{*}(s)}{|s|} \leq \lambda^l.\]
It remains to prove that the limits \(\lim\limits_{|s|\to+\infty} \frac{\rho^*(s)}{|s|}\) exist. We focus on the rightward limit \(s\to+\infty\), the leftward case follows similarly.

By Taylor inequality, one can verify that $H(s,p,q)\ge Ap^2+q-\mu+\mu R(s)e^{\tau q}$, where $A=\frac{d}{2}\int_{\mathbb{R}}J(y)y^2dy$. Using the definition of viscosity subsolution together with $\rho^*\ge0$, we obtain in the viscosity sense 
\[
-s(\rho^*)'+A|(\rho^*)'|^2-\mu\le H\left(s,(\rho^*)',\rho^*-s(\rho^*)'\right)\le0.
\]
Thus $\rho^*\in {\rm Lip}_{{\rm loc}}([0,+\infty))$. Denote $\sigma(s)=\frac{\rho^*(s)}{s}$, then $\sigma(s)$ is a viscosity subsolution of 
\[
\min\left\{\sigma,-s^2\sigma'+d\left(\int_{\mathbb{R}}J(y)e^{(s\sigma'+\sigma)y}dy-1\right)-\mu+\mu R(s)e^{-s^2\tau\sigma'}\right\}=0.
\]
Suppose that $\sigma$ attains a positive local maximum at some $s_0>0$. Taking the test function $\varphi\equiv0$ in the viscosity subsolution condition yields 
\[
-\mu+\mu R(s_0)<d\left(\int_{\mathbb{R}}J(y)e^{\sigma(s_0)y}dy-1\right)-\mu+\mu R(s_0)\le0,
\]
which is a contradiction. Thus $\sigma(s)$ has no positive local maximum point in $(0,+\infty)$. This implies that $\lim\limits_{s\to+\infty}\frac{\rho^{*}(s)}{s}$ exist. We completes the proof.
\end{proof}

\begin{lemma}\label{lem:u^vareps>0}
Assume that {\rm(J), (F1)} hold and {\rm(\text{IC}$^\lambda)$} holds for some $\lambda\in(0,+\infty]$. Then
\[
\liminf_{\varepsilon\to0}u^\varepsilon>0 \text{ locally uniformly in } \mathcal{B}=\text{Int}\,\{(t,x)\in(0,\infty)\times\mathbb{R}:w^*(t,x)=0\}.
\]
\end{lemma}
\begin{proof}
We proceed by contradiction. Fix a compact subset \(Q_4 \subset \subset \mathcal{B}\) and a point \((t_0, x_0) \in Q_4\). Suppose that there exist sequences \(\varepsilon_n \to 0\) and \((t_{\varepsilon_n}, x_{\varepsilon_n}) \to (t_0, x_0)\) such that $u^{\varepsilon_n}(t_{\varepsilon_n}, x_{\varepsilon_n}) \to 0$ as $n \to \infty$.
For simplicity, we denote these sequences by \(\varepsilon\) and \((t_\varepsilon, x_\varepsilon)\).

Define the test function
\[
\varphi_\varepsilon(t, x) := |t - t_\varepsilon|^2 + |x - x_\varepsilon|^2.
\]
Fix a constant $r_3>0$ sufficiently small.
Since \(w^\varepsilon(t, x) \to 0\) uniformly on a neighborhood \(B_{2r_3}(t_0, x_0)\) as \(\varepsilon \to 0\), we may choose \(\varepsilon\) sufficiently small so that \((t_\varepsilon, x_\varepsilon) \in B_{2r_3}(t_0, x_0)\). Then the function \(w^\varepsilon - \varphi_\varepsilon\) attains an interior maximum at some point \((t'_\varepsilon, x'_\varepsilon) \in B_{2r_3}(t_0, x_0)\). Observe that
\begin{equation}\label{eq:conv}
(t'_\varepsilon, x'_\varepsilon) \to (t_0, x_0) \quad \text{and} \quad (t_\varepsilon, x_\varepsilon) \to (t_0, x_0) \quad \text{as } \varepsilon \to 0.
\end{equation}
By construction, we have
\begin{align}\label{eq:ineq-w}
w^\varepsilon(t'_\varepsilon, x'_\varepsilon) \ge (w^\varepsilon - \varphi_\varepsilon)(t'_\varepsilon, x'_\varepsilon)\ge (w^\varepsilon - \varphi_\varepsilon)(t_\varepsilon, x_\varepsilon) = w^\varepsilon(t_\varepsilon, x_\varepsilon).
\end{align}
Recall that \(u^\varepsilon(t, x) = \exp\left(-\frac{w^\varepsilon(t, x)}{\varepsilon}\right)\). Hence \eqref{eq:ineq-w} implies $0 \le u^\varepsilon(t'_\varepsilon, x'_\varepsilon) \le u^\varepsilon(t_\varepsilon, x_\varepsilon)$.
Since \(u^\varepsilon(t_\varepsilon, x_\varepsilon) \to 0\) by assumption, we obtain
\begin{equation}\label{eq:u-eps-prime}
u^\varepsilon(t'_\varepsilon, x'_\varepsilon) \to 0 \quad \text{as } \varepsilon \to 0.
\end{equation}
Moreover, from the definition of \(\varphi_\varepsilon\) and \eqref{eq:conv}, we have (after passing to a subsequence if necessary)
\begin{equation}\label{eq:partial-varphi}
\partial_t\varphi_\varepsilon(t'_\varepsilon, x'_\varepsilon) \to 0,\quad
\partial_x\varphi_\varepsilon(t'_\varepsilon, x'_\varepsilon) \to 0,\quad
|D\varphi_\varepsilon(t'_\varepsilon - \varepsilon\tau, x'_\varepsilon)| \to 0.
\end{equation}

We claim that
\begin{equation}\label{eq:claim}
u^\varepsilon(t'_\varepsilon - \varepsilon\tau, x'_\varepsilon) \to 0 \quad \text{as } \varepsilon \to 0.
\end{equation}
Suppose, to the contrary, that there exist \(\tilde{\delta}>0\) and a subsequence (still denoted by \(\varepsilon\)) such that $u^\varepsilon(t'_\varepsilon - \varepsilon\tau, x'_\varepsilon) \ge \tilde{\delta}$.
By hypothesis (F1), there exists a constant \(C_{\tilde{\delta}}>0\) such that
\[
f^\varepsilon\left(x'_\varepsilon - c_1 t'_\varepsilon, u^\varepsilon(t'_\varepsilon - \varepsilon\tau, x'_\varepsilon)\right) \ge C_{\tilde{\delta}}.
\]
Using \eqref{eq:u-eps-prime}, we can choose \(\tilde{M}>0\) large enough so that
\[
\frac{f^\varepsilon\left(x'_\varepsilon - c_1 t'_\varepsilon, u^\varepsilon(t'_\varepsilon - \varepsilon\tau, x'_\varepsilon)\right)}
{u^\varepsilon(t'_\varepsilon, x'_\varepsilon)} \ge
\frac{C_{\tilde{\delta}}}{u^\varepsilon(t'_\varepsilon, x'_\varepsilon)} \ge \tilde{M}
\]
for all sufficiently small \(\varepsilon\). Now, \eqref{w^varepsilon} gives
\begin{align*}
0 ={}& d\int_{\mathbb{R}} J(y)\left[e^{-\frac{w^\varepsilon(t'_\varepsilon, x'_\varepsilon-\varepsilon y)-w^\varepsilon(t'_\varepsilon, x'_\varepsilon)}{\varepsilon}} - 1\right]\,dy \\
      &+ \partial_t w^\varepsilon(t'_\varepsilon, x'_\varepsilon) - \mu
      + \mu\,\frac{f^\varepsilon\left(x'_\varepsilon - c_1 t'_\varepsilon, u^\varepsilon(t'_\varepsilon - \varepsilon\tau, x'_\varepsilon)\right)}
           {u^\varepsilon(t'_\varepsilon, x'_\varepsilon)} \\
\ge{}& -d + \partial_t\varphi_\varepsilon(t'_\varepsilon, x'_\varepsilon) - \mu + \mu\tilde{M},
\end{align*}
where we used that \((t'_\varepsilon, x'_\varepsilon)\) is a maximum point of \(w^\varepsilon - \varphi_\varepsilon\). Letting \(\varepsilon\to0\) and using \eqref{eq:partial-varphi}, we obtain $0 \ge -d - \mu + \mu\tilde{M},$
which is impossible for \(\tilde{M}\) sufficiently large. Hence the claim \eqref{eq:claim} holds.

Choose \(0<\tilde{\eta}<1\) such that
\begin{equation}\label{eq:eta}
\underline{R}\left(\frac{x_0}{t_0}\right) > 1 + 2\tilde{\eta}.
\end{equation}
Such an \(\tilde{\eta}\) exists because \(\underline{R}(x_0/t_0) > 1\) by \eqref{under-R}. For small \(\varepsilon>0\), it holds that
\begin{align*}
&\mu\left(\underline{R}(x_0/t_0) - 2\tilde{\eta}\right) + o(1) \\
\le& d\int_{\mathbb{R}} J(y)\left[e^{\frac{\varphi_\varepsilon(t'_\varepsilon, x'_\varepsilon) -
          \varphi_\varepsilon(t'_\varepsilon, x'_\varepsilon-\varepsilon y)}{\varepsilon}} - 1\right]\,dy
          + \partial_t\varphi_\varepsilon(t'_\varepsilon, x'_\varepsilon)  \\
&\qquad+ \mu\left(\underline{R}(x_0/t_0) - 2\tilde{\eta}\right)\,
          e^{\frac{\varphi_\varepsilon(t'_\varepsilon, x'_\varepsilon) -
          \varphi_\varepsilon(t'_\varepsilon - \varepsilon\tau, x'_\varepsilon)}{\varepsilon}} \\
\le& d\int_{\mathbb{R}} J(y)\left[e^{\frac{w^\varepsilon(t'_\varepsilon, x'_\varepsilon) -
          w^\varepsilon(t'_\varepsilon, x'_\varepsilon-\varepsilon y)}{\varepsilon}} - 1\right]\,dy
          + \partial_t w^\varepsilon(t'_\varepsilon, x'_\varepsilon)   \\
&\qquad+\mu\left(\partial_u f^\varepsilon(x'_\varepsilon - c_1 t'_\varepsilon, 0) - 2\tilde{\eta}\right)\,
          e^{\frac{w^\varepsilon(t'_\varepsilon, x'_\varepsilon) -
          w^\varepsilon(t'_\varepsilon - \varepsilon\tau, x'_\varepsilon)}{\varepsilon}} \\
\le& d\int_{\mathbb{R}} J(y)\left[e^{\frac{w^\varepsilon(t'_\varepsilon, x'_\varepsilon) -
          w^\varepsilon(t'_\varepsilon, x'_\varepsilon-\varepsilon y)}{\varepsilon}} - 1\right]\,dy
          + \partial_t w^\varepsilon(t'_\varepsilon, x'_\varepsilon) + \mu\,
          \frac{f^\varepsilon\left(x'_\varepsilon - c_1 t'_\varepsilon,
          u^\varepsilon(t'_\varepsilon - \varepsilon\tau, x'_\varepsilon)\right)}
          {u^\varepsilon(t'_\varepsilon, x'_\varepsilon)} \\
=& \mu,
\end{align*}
where the first inequality follows from \eqref{eq:partial-varphi}, while the second inequality is due to the fact that \((t'_\varepsilon, x'_\varepsilon)\) is a maximum point of \(w^\varepsilon - \varphi_\varepsilon\); the third inequality relies on (F1) and \eqref{eq:claim}, and the final equality comes from the equation for \(w^\varepsilon\). Taking the limit as \(\varepsilon\to0\), we obtain
\[
\mu\left(\underline{R}(x_0/t_0) - 2\tilde{\eta}\right) \le \mu,
\]
which contradicts \eqref{eq:eta} since \(\underline{R}(x_0/t_0) - 2\tilde{\eta} > 1\). This contradiction shows that our initial assumption was false, and the proof is complete.
\end{proof}

By Lemmas \ref{lem:rho-sup-sol} and \ref{lem:rho-sub-sol}, $\rho^*$ and $\rho_*$ are respectively a viscosity subsolution and supersolution of \eqref{rho1}. Moreover,
\[
\rho^*(0)=\rho_*(0)=0\quad \text{ and }\quad\lim_{|s|\to\infty}\frac{\rho^*(s)}{|s|}\le\liminf_{|s|\to\infty}\frac{\rho_*(s)}{|s|}.
\]
The comparison principle for \eqref{rho1} (see \cite[Prop. 2.10-2.11]{Lam2022}\cite[Prop. 2.4]{Tao2026}) yields $\rho^*(s)\le\rho_*(s)$ for all $s\in\mathbb{R}$. As $\rho^*\ge\rho_*$ by construction, we obtain $\rho^*\equiv\rho_*:=\rho$. Consequently, $w^\varepsilon(t,x)\to t\rho(t,x)$ locally uniformly in $(0,\infty)\times\mathbb{R}$. Combined with Lemmas \ref{lem:u^vareps=0} and \ref{lem:u^vareps>0}, we have the following results.

\begin{corollary}\label{col:rho-boundary-point-r}
Assume that {\rm(\text{IC}$^\lambda)$} holds for some $\lambda\in(0,+\infty]$.
Let $\rho\in C([0,\infty);[0,\infty))$ be the unique viscosity solution of \eqref{rho1}
in $(0,\infty)$ satisfying
\begin{equation}\label{rho-boundary-1}
    \rho(0) = 0, \qquad \lim_{s\to+\infty} \frac{\rho(s)}{s} = \lambda^r \in (0,+\infty].
\end{equation}
Then there exists a unique $s_r\in[0,+\infty)$ such that
\[
\rho(s)=0\quad\text{for }s\in[0,s_r],\qquad \rho(s)>0\quad\text{for }s\in(s_r,\infty),
\]
and the following limits hold:
\begin{align*}
\begin{cases}
    \lim\limits_{t\to\infty} \sup\limits_{(s_r+\eta_1)t \leq x \leq (s_r+\eta_2)t} u(t,x) &= 0
        \qquad\text{for any } 0<\eta_1<\eta_2,\\
    \liminf\limits_{t\to\infty} \inf\limits_{(s_r-\eta_2)t \leq x \leq (s_r-\eta_1)t} u(t,x) &> 0
        \qquad\text{for any } 0<\eta_1<\eta_2<s_r.
\end{cases}
\end{align*}
\end{corollary}
\begin{proof}
Set $Z=\{s\ge0:\rho(s)=0\}$. Then $Z\neq\emptyset$ and $s_r:=\sup Z<+\infty$
by \eqref{rho-boundary-1}. If $\rho$ attained a positive local maximum at some $s_0>0$,
taking $\varphi\equiv0$ in the viscosity subsolution condition yields
$H(s_0,0,\rho(s_0))\le0$, contradicting
\[
H(s_0,0,\rho(s_0))=\rho(s_0)-\mu+\mu R(s_0)e^{\tau\rho(s_0)}
>\rho(s_0)-\mu+\mu=\rho(s_0)>0.
\]
Thus $\rho$ has no positive local maximum, which implies $\rho$ is nondecreasing on $[0,\infty)$.
Consequently, $\rho(s) = 0$ for $s\in[0,s_r]$, $\rho(s) > 0$ for $s\in(s_r,\infty)$.
The stated limits now follow from Lemmas~\ref{lem:u^vareps=0} and~\ref{lem:u^vareps>0}
via the convergence $w^\varepsilon(t,x)\to t\rho(x/t)$. This ends the proof.
\end{proof}

\begin{corollary}\label{col:rho-boundary-point-l}
Assume that {\rm(\text{IC}$^\lambda)$} holds for some $\lambda\in(0,+\infty]$.
Let $\rho\in C((-\infty,0];[0,\infty))$ be the unique viscosity solution of \eqref{rho1}
in $(-\infty,0)$ satisfying
\begin{equation}\label{rho-boundary-2}
    \rho(0) = 0, \qquad \lim_{s\to-\infty} \frac{\rho(s)}{|s|} = \lambda^l \in (0,+\infty].
\end{equation}
Then there exists a unique $s_l \le 0$ such that
\[
\rho(s) = 0 \ \text{ for } s\in[s_l,0], \qquad \rho(s) > 0 \ \text{ for } s\in(-\infty,s_l),
\]
and the following limits hold:
\begin{align*}
\begin{cases}
    \lim\limits_{t\to\infty} \sup\limits_{(s_l-\eta_2)t \leq x \leq (s_l-\eta_1)t} u(t,x) &= 0
        \qquad\text{for any } 0<\eta_1<\eta_2,\\
    \liminf\limits_{t\to\infty} \inf\limits_{(s_l+\eta_1)t \leq x \leq (s_l+\eta_2)t} u(t,x) &> 0
        \qquad\text{for any } 0<\eta_1<\eta_2<-s_l.
\end{cases}
\end{align*}
\end{corollary}
\begin{proof}
The proof is similar to that of Corollary~\ref{col:rho-boundary-point-r} and is omitted.
\end{proof}

\section{Explicit formulas of the spreading speed}\label{Sec:Speed}
\noindent

In this section, we construct a series of viscosity solutions to derive the explicit formulas of spreading speeds.

\begin{lemma}\label{lem:H-convex}
Assume that $\mathrm{(J)}$ holds. For each $s \in \mathbb{R}$, there exists $\tilde{H}(s,p)$ such that the following holds.
\begin{itemize}
\item[(i)] $\operatorname{sgn} H(s,p,q) = \operatorname{sgn}\big(q + \tilde{H}(s,p)\big)$ for all $(q,p) \in \mathbb{R}^2$.
\item[(ii)] $p \mapsto \tilde{H} \in C^\infty(\mathbb{R})$.
\item[(iii)] \(p \mapsto \tilde{H}(s,p)\) is convex (resp. strictly convex) if
\((p,q) \mapsto H(s,p,q)\) is convex (resp. strictly convex), and it is even if
\(H(s,-p,q) = H(s,p,q)\) for all \(p,q\).
\end{itemize}
\end{lemma}
\begin{proof}
(i) For fixed $s$, the function $q\mapsto H(s,p,q)$ is continuous and strictly increasing
because $\partial_q H = 1+\mu R(s)\tau e^{\tau q} > 0$. Moreover, 
\[
\lim\limits_{q\to +\infty} H(s,p,q)=+\infty,\quad \lim\limits_{q\to -\infty} H(s,p,q)=-\infty.
\]
Hence, for each $p\in\mathbb{R}$ there exists a unique $q_0(s,p)$ such that
$H(s,p,q_0)=0$. Setting $\tilde{H}(s,p):=-q_0(s,p)$ we have the equivalence
\[
H(s,p,q)=0 \;\Longleftrightarrow\; q=-\tilde{H}(s,p).
\]
By strict monotonicity, $q>-\tilde{H}(s,p)$ implies $H(s,p,q)>0$ and
$q<-\tilde{H}(s,p)$ implies $H(s,p,q)<0$. Thus
$\operatorname{sgn} H(s,p,q) = \operatorname{sgn}\big(q+\tilde{H}(s,p)\big)$
for all $(q,p)\in\mathbb{R}^2$.

(ii) The $C^\infty$ regularity of $\tilde{H}$ in $p$ follows directly from the implicit function theorem applied to $H(s,p,q)=0$ at the unique solution, since $\partial_q H > 0$ and $H$ is smooth in $(p,q)$ by {\rm(J)}.

(iii) Assume $(p,q) \mapsto H(s,p,q)$ is convex.
Let $p_1,p_2\in\mathbb{R}$ and set $q_i=-\tilde{H}(s,p_i)$, so that
$H(s,p_i,q_i)=0$ for $i=1,2$. By convexity,
\begin{align}\label{eq:H-convex}
H\!\left(s,\frac{p_1+p_2}{2},\frac{q_1+q_2}{2}\right)
\le \frac12 H(s,p_1,q_1)+\frac12 H(s,p_2,q_2)=0.
\end{align}
Since $H(s,p,q)=0$ exactly at $q=-\tilde{H}(s,p)$ and $H$ is strictly increasing
in $q$, the inequality $H(s,\frac{p_1+p_2}{2},\frac{q_1+q_2}{2})\le 0$ forces
\[
\frac{q_1+q_2}{2} \le -\tilde{H}\!\left(s,\frac{p_1+p_2}{2}\right),
\]
which is precisely the convexity of $\tilde{H}(s,\cdot)$.
If, in addition, $H$ is strictly convex and $(p_1,q_1)\neq(p_2,q_2)$, then
the inequality in \eqref{eq:H-convex} is strict, yielding strict convexity of $\tilde{H}(s,\cdot)$.

If $H(s,-p,q)=H(s,p,q)$, from the definition of $\tilde{H}$ we obtain
\[
H(s,-p,-\tilde{H}(s,-p))=0 = H(s,p,-\tilde{H}(s,p)) = H(s,-p,-\tilde{H}(s,p)),
\]
and uniqueness of the root forces $\tilde{H}(s,-p)=\tilde{H}(s,p)$. This ends the proof.
\end{proof}

Thus, although the original Hamiltonian $H$ couples the spatial and temporal derivatives, its monotonicity in $q$ permits the reduction to an effective Hamiltonian $\tilde{H}$ that governs the dynamics in a standard form.
For the lower and upper semicontinuous
envelopes, the following useful equivalences are established in \cite[Lemma A.2]{Lam2022}:
\[
\begin{aligned}
H^*(s_0, p_0, q_0) \ge 0 &\;\Longleftrightarrow\; q_0 + \tilde{H}^*(s_0, p_0) \ge 0,\\
H_*(s_0, p_0, q_0) \le 0 &\;\Longleftrightarrow\; q_0 + \tilde{H}_*(s_0, p_0) \le 0.
\end{aligned}
\]

Let \(\mathbf{H}_\pm(p):\mathbb{R}\to(0,\infty)\) be uniquely defined by the implicit formula
\[
d\Bigl(\int_{\mathbb{R}}J(y)e^{py}dy-1\Bigr) - \mathbf{H}_\pm(p) -\mu+\mu R_\pm e^{-\tau \mathbf{H}_\pm(p)} = 0,
\]
and set \(c_\pm(p)=\frac{\mathbf{H}_\pm(p)}{p}\).
Then \eqref{rho1} can be re-written as
\[
\min\Bigl\{ \rho,\; \rho-s\rho'+\tilde{H}(s, \rho') \Bigr\} = 0,
\qquad\text{where}\qquad
\tilde{H}(s, p) =
\begin{cases}
\mathbf{H}_-(p) & \text{for } s \le c_1,\\[2pt]
\mathbf{H}_+(p) & \text{for } s > c_1.
\end{cases}
\]

\begin{corollary}\label{lem:gamma-pm}
Assume that {\rm (J)} and {\rm(F2)} hold.
The functions $\mathbf{H}_+$ and $\mathbf{H}_-$ satisfy the following properties.
\begin{enumerate}
    \item[(i)] Both $\mathbf{H}_+$ and $\mathbf{H}_-$ are well-defined, even, and strictly convex.
    \item[(ii)] Both $\mathbf{H}_+(p)$ and $\mathbf{H}_-(p)$ grow super-linearly as $|p|\to\infty$.
    \item[(iii)] There exist unique constants $\nu^*_+$, $\nu^*_->0$ such that
          \[
          c_\pm(p)\ge c_\pm(\nu^*_\pm) \quad\text{ for all } p>0,
          \]
          with equality holding if and only if $p=\nu^*_\pm$, respectively. Moreover, $c_\pm(p)$ are strictly decreasing on $(0,\nu^*_\pm)$, and strictly increasing on $(\nu^*_\pm,\infty)$, respectively.
\end{enumerate}
\end{corollary}
\begin{proof}
Assertion (i) is a direct result of Lemma~\ref{lem:H-convex} by {\rm (J)} and {\rm (F2)}.

We next show (ii).  We prove for $\mathbf{H}_+$; the case for $\mathbf{H}_-$ is identical.
By assumption (J), there exist $\hat{\delta}>\hat{r}>0$ such that $J>0$ on a subset of $B_{\hat{r}}(\hat{\delta})$ with positive measure. Hence for $p>0$,
\begin{align*}
\int_{\mathbb{R}}J(y)e^{py}dy
&\ge\int_{B_{\hat{r}}(\hat{\delta})}J(y)e^{py}dy \ge e^{p(\hat{\delta}-\hat{r})}\int_{B_{\hat{r}}(\hat{\delta})}J(y)dy \ge \hat{C} e^{(\hat{\delta}-\hat{r})p},
\end{align*}
with $\hat{C}>0$. Consequently,
\[
\mathbf{H}_+(p)=d\left(\int_{\mathbb{R}}J(y)e^{py}dy-1\right)-\mu+\mu R_+ e^{-\tau\mathbf{H}_+(p)}\ge d\hat{C} e^{(\hat{\delta}-\hat{r})p}-d-\mu,
\]
which implies that $\mathbf{H}_+(p)$ grows super-linearly at $\infty$. By evenness, the same holds at $-\infty$.

We now prove (iii). Since $\mathbf{H}_+$ is even, we have $\mathbf{H}_+'(0)=0$.
We first note $\mathbf{H}_+(0)>0$; otherwise evaluating $\Delta_+(0,\mathbf{H}_+(0))=0$
leads to
\[
-\mathbf{H}_+(0)-\mu+\mu R_+ e^{-\tau\mathbf{H}_+(0)}\ge\mu(R_+-1)>0,
\]
a contradiction.
For $p>0$, let $c_+(p)=\mathbf{H}_+(p)/p$. Then $c_+'(p)=\frac{p\mathbf{H}_+'(p)-\mathbf{H}_+(p)}{p^2}.$
The numerator $g_+(p)=p\mathbf{H}_+'(p)-\mathbf{H}_+(p)$ satisfies 
\[
g_+(0)=-\mathbf{H}_+(0)<0, \quad g_+'(p)=p\mathbf{H}_+''(p)>0,
\]
hence $g_+(p)$ is strictly increasing with a unique zero $\nu^*_+>0$.
Thus $c_+(p)$ attains its minimum at $\nu^*_+$ in $p>0$. The case of $\mathbf{H}_-$ is identical. This completes the proof.
\end{proof}

As a consequence of Corollary~\ref{lem:gamma-pm}, the derivatives
$\mathbf{H}_\pm'$ are strictly increasing on $\mathbb{R}$ and, by
super-linear growth, tend to $\pm\infty$ as $p\to\pm\infty$; hence
they are homeomorphisms. Let $\Psi_\pm$ denote their inverses and set
$c_\pm^* = \inf\limits_{\nu>0} c_\pm(\nu) = c_\pm(\nu_\pm^*) > 0$.
Define the Lagrangians
\[
\mathbf{L}_\pm(v) = \max_{p\in\mathbb{R}} [vp - \mathbf{H}_\pm(p)]
= v\Psi_\pm(v) - \mathbf{H}_\pm(\Psi_\pm(v)), \qquad v\in\mathbb{R}.
\]
Since $R_- < R_+$ and $\partial_\mathbf{H}\Delta_\pm < 0$, we have
$\mathbf{H}_-(p) < \mathbf{H}_+(p)$ for $p>0$, which directly implies
$c_-^* < c_+^*$ and, for every $c_1>0$, $\mathbf{L}_-(c_1) > \mathbf{L}_+(c_1)$.
If $c_1 > c_+^*$, then
\begin{align}\label{eq:L_+0}
\mathbf{L}_+(c_1) \ge c_1\nu_+^* - \mathbf{H}_+(\nu_+^*)
> c_+^*\nu_+^* - \mathbf{H}_+(\nu_+^*) = 0.
\end{align}

Now introduce the sets
\begin{align*}
E_1 &:= \{c_1 : c_1 \ge c_+^*\},\\
E_2 &:= \{(\lambda^r,c_1) : c_1\lambda^r > \mathbf{H}_+(\lambda^r)
\text{ for } 0<\lambda^r\le\nu_+^*\}  \cup \{(\lambda^r,c_1) : c_1 \ge \mathbf{H}_+'(\lambda^r)
\text{ for } \lambda^r > \nu_+^*\}.
\end{align*}
Note that $\mathbf{H}_+'(p) > c_+(p)$ for $p > \nu_+^*$; thus for
$\lambda^r > \nu_+^*$, the condition $c_1 \ge \mathbf{H}_+'(\lambda^r)$
already yields $c_1\lambda^r > \mathbf{H}_+(\lambda^r)$. Hence
\begin{align}\label{eq:E_2subset}
E_2 \subset \{(\lambda^r,c_1) : c_1\lambda^r > \mathbf{H}_+(\lambda^r),\,
\lambda^r > 0\}.
\end{align}

\begin{lemma}\label{lem:bar-p}
For any $c_1 \in E_1$, the equation in $p$
\begin{align*}\label{eq:bar-p}
c_1 p - \mathbf{H}_-(p) = \mathbf{L}_+(c_1)
\end{align*}
has a smallest positive root $\bar{p}:=\bar{p}(c_1)$.
Then $\bar{p}(c_1) < \min\{\Psi_+(c_1),\Psi_-(c_1)\}$, and $\bar{p}$ is strictly increasing in $c_1$.
Moreover, there exists a unique $\bar{c}_1 > c_+^*$ such that $\bar{p}(\bar{c}_1) = \nu_-^*$.
\end{lemma}
\begin{proof}
Define an auxiliary function
\begin{align*}
F_1(c_1,p) = c_1p - \mathbf{H}_-(p) - \mathbf{L}_+(c_1) \quad\text{ for }  c_1 \in E_1.
\end{align*}
Then $F_1$ is increasing in $p$ on $[0, \Psi_-(c_1)]$ and decreasing on $[\Psi_-(c_1), \infty)$. Moreover,
\[
\begin{aligned}
F_1(c_1, 0) \leq -\mathbf{H}_-(0) < 0 , \quad F_1(c_1, \Psi_+(c_1)) = -\mathbf{H}_-(\Psi_+(c_1)) + \mathbf{H}_+(\Psi_+(c_1))>0.
\end{aligned}
\]
It then follows that the smallest root $\bar{p}(c_1) \in (0, \min\{\Psi_+(c_1),\Psi_-(c_1)\})$. Furthermore,
\begin{equation}\label{partial_F-1}
\begin{aligned}
\partial_{c_1} F_1(c_1,\bar{p}) = \bar{p} - \Psi_+(c_1) < 0,\,\,\partial_p F_1(c_1,\bar{p})= c_1-\mathbf{H}_-'(\bar{p})> 0,
\end{aligned}
\end{equation}
where the last inequality follows from the strict convexity of $\mathbf{H}_-$ (which makes $\mathbf{H}_-'$ strictly increasing) together with $\bar{p}<\Psi_-(c_1)$, so that $\mathbf{H}_-'(\bar{p}) < \mathbf{H}_-'(\Psi_-(c_1)) = c_1$.
Differentiating $F_1(c_1,p)=0$ with respect to $c_1$ along the curve $p=\bar{p}(c_1)$ gives
\[
\frac{\partial F_1}{\partial c_1}+\frac{\partial F_1}{\partial p}\frac{\mathrm{d}\bar{p}}{\mathrm{d} c_1}=0,
\]
which, combined with \eqref{partial_F-1}, implies that \( \bar{p}(c_1) \) is strictly increasing in \( c_1 \). Since \( F_1(c_+^*,\nu_-^*) = \nu_-^*(c_+^* - c_-^*) > 0 \), we have \( \bar{p}(c_+^*) < \nu_-^* \). Besides, the super-linearity of $\mathbf{H}_-$ implies \( \bar{p}(+\infty) = +\infty \). Thus, there exists a unique \( \bar{c}_1>c_+^* \) such that \( \bar{p}(\bar{c}_1) = \nu_-^* \).
\end{proof}

\begin{lemma}\label{lem:under-p}
For any $(\lambda^r,c_1) \in E_2$, the equation in $p$
\begin{align}\label{eq:under-p}
c_1 p - \mathbf{H}_-(p) = c_1 \lambda^r - \mathbf{H}_+(\lambda^r)
\end{align}
has a smallest positive root $\underline{p}:=\underline{p}(\lambda^r,c_1)$.
Then $\underline{p}(\lambda^r,c_1) < \min\{\Psi_-(c_1),\lambda^r\}$.
Furthermore,
\begin{itemize}
    \item[(i)] If $\lambda^r \in (0,\nu_-^*]$, then $\underline{p} < \nu_-^*$ for all $c_1>c_+(\lambda^r)$.
    \item[(ii)] $\underline{p}$ is strictly increasing in $c_1$, and also increasing in $\lambda^r$ whenever $c_1 > \mathbf{H}_+'(\lambda^r)$.
\end{itemize}
\end{lemma}
\begin{proof}
Define an auxiliary function
\begin{align*}
F_2(\lambda^r,c_1, p) = c_1p - \mathbf{H}_-(p) - c_1\lambda^r + \mathbf{H}_+(\lambda^r) \quad\text{ for } (\lambda^r, c_1) \in E_2.
\end{align*}
Then $F_2$ is increasing in $p$ on $[0, \Psi_-(c_1)]$ and decreasing on $[\Psi_-(c_1), \infty)$ by concavity and \eqref{eq:L_+0}. Moreover,
\[
\begin{aligned}
F_2(\lambda^r,c_1, 0) = -\mathbf{H}_-(0) - c_1\lambda^r + \mathbf{H}_+(\lambda^r) < 0, \quad F_2(\lambda^r,c_1,\lambda^r) = -\mathbf{H}_-(\lambda^r) + \mathbf{H}_+(\lambda^r) > 0,
\end{aligned}
\]
where we used \eqref{eq:E_2subset}.
It then follows that the smallest root $\underline{p}( \lambda^r, c_1) \in (0,\min\{\lambda^r,\Psi_-(c_1)\})$. In particular, if $\lambda^r \in (0,\nu_-^*]$, then $\underline{p} <\lambda^r\le\nu_-^*$ for all $c_1>c_+(\lambda^r)$.
Additionally,
\begin{equation}\label{partial_F-2}
\begin{aligned}
&\partial_{\lambda^r} F_2(\lambda^r,c_1, \underline{p}) = -c_1 + \mathbf{H}_+'(\lambda^r), \quad\partial_{c_1} F_2(\lambda^r,c_1,\underline{p}) = \underline{p} - \lambda^r < 0,\\
&\partial_p F_2(\lambda^r,c_1, \underline{p})=c_1-\mathbf{H}_-'(\underline{p}) > 0.
\end{aligned}
\end{equation}
where the last inequality follows from the strict convexity of $\mathbf{H}_-$ and $\underline{p}<\Psi_-(c_1)$.
Differentiating the identity $F_2(\lambda^r,c_1,p)=0$ with respect to $c_1$ and $\lambda^r$ along the curve $p=\underline{p}(\lambda^r,c_1)$ yields
\[
\frac{\partial F_2}{\partial c_1}+\frac{\partial F_2}{\partial p}\frac{\partial\underline{p}}{\partial c_1}=0, \quad \frac{\partial F_2}{\partial \lambda^r}+\frac{\partial F_2}{\partial p}\frac{\partial\underline{p}}{\partial \lambda^r}=0.
\]
Using \eqref{partial_F-2}, we see that \(\underline{p}(\lambda^r, c_1)\) is increasing in \(c_1\), and also increasing in \(\lambda^r\) when \(c_1 > \mathbf{H}_+'(\lambda^r)\).
\end{proof}

For $(\lambda^r,c_1)\in E_2$ with $\lambda^r>\nu_-^*$, consider the equation $F_2(\lambda^r,c_1,\nu_-^*)=0$.
For each fixed $\lambda^r$, note that $F_2(\lambda^r,c_1,\nu_-^*) \to \mp\infty$ as $c_1 \to \pm\infty$ and $F_2$ is strictly decreasing in $c_1$ (see \eqref{partial_F-2}). This defines a function $c_1 = g(\lambda^r)$, which solves $F_2 = 0$ explicitly, and hence
\begin{align}\label{eq:g}
 g(\lambda^r) = \frac{\mathbf{H}_+(\lambda^r) - \mathbf{H}_-(\nu_-^*)}{\lambda^r - \nu_-^*}, \qquad \lambda^r > \nu_-^*.
\end{align}

\begin{lemma}\label{lem:g(lambda^r)}
Let $g$ be defined by \eqref{eq:g}.
Assume that $(\lambda^r,c_1)\in E_2$ such that $\underline{p}$ is well-defined, then the following assertions hold true.
\begin{itemize}
    \item[(i)] For any $\lambda^r \in (\nu_-^*,\Psi_+(\bar{c}_1)]$, we have
    \[
    \underline{p}(\lambda^r,c_1) \leq\nu_-^* \quad\Longleftrightarrow\quad c_1 \leq g(\lambda^r),
    \]
and the equality holds iff $c_{1}=g(\lambda^{r})$.
    \item[(ii)] $g$ is strictly decreasing on $(\nu_-^*,\Psi_+(\bar{c}_1)]$.
    \item[(iii)] The constant $\bar{c}_1$ from Lemma \ref{lem:bar-p} satisfies $\bar{c}_1 = g\left(\Psi_+(\bar{c}_1)\right)$.
\end{itemize}
\end{lemma}

\begin{proof}
We first prove (i). If $\lambda^r\in(\nu_-^*,\Psi_+(\bar{c}_1)]$, then
\begin{align*}
c_1 < g(\lambda^r) \iff& c_1\lambda^r-\mathbf{H}_+(\lambda^r)<c_1\nu_-^*-\mathbf{H}_-(\nu_-^*)\\
\iff& c_1\underline{p}-\mathbf{H}_-(\underline{p})<c_1\nu_-^*-\mathbf{H}_-(\nu_-^*)\\
\iff& \underline{p}<\nu_-^*,
\end{align*}
where the last inequality follows from the strict monotonicity of function $h_1(s)=c_1s-\mathbf{H}_-(s)$ for $s\in(-\infty,\Psi_-(c_1)]$ and $0<\underline{p}<\Psi_-(c_1)$. The case for $c_1 = g(\lambda^r)$ is analogous.

We then prove (ii). By (i), we have $F_2(\lambda^r,c_1,\nu_-^*)|_{c_1=g(\lambda^r)}\equiv0$ for all $\lambda^r$. Differentiating in $\lambda^r$ yields
\[
\frac{\partial F_2}{\partial c_1}\,\frac{\mathrm{d}g}{\mathrm{d}\lambda^r} + \frac{\partial F_2}{\partial \lambda^r}=0.
\]
From \eqref{partial_F-2}, we obtain \(g'(\lambda^r)<0\) for \(\lambda^r \in (\nu_-^*,\Psi_+(\bar{c}_1)]\).

Finally, we prove (iii), a direct computation shows $F_2\left(\Psi_+(\bar{c}_1),\bar{c}_1, \nu_-^*\right) = F_1\left(\bar{c}_1, \nu_-^*\right) = 0,$
which gives that \(\bar{c}_1 = g\left(\Psi_+(\bar{c}_1)\right)\).
\end{proof}

\subsection{Rightward Spreading Speed}
\noindent

In what follows, adapting the constructions from Section~3 of \cite{Lam2022}, we construct explicit viscosity solutions for the Hamilton-Jacobi equation and then derive the explicit formulas for the rightward spreading speeds.

\begin{lemma}\label{lem:vis-A1}
If $(\lambda^r,c_1)\in A_1$, then $c_r=c_+(\lambda^r)$.
\end{lemma}
\begin{proof}
Define
\begin{align}\label{rho11}
\rho(s)=\max\left\{0,\;\lambda^r s-\mathbf{H}_+(\lambda^r)\right\},\qquad s\in[0,+\infty).
\end{align}
One can check directly that $\rho$ defined by \eqref{rho11}  satisfies \eqref{rho-boundary-1}, and is a classical solution of \eqref{rho1} at every point except \(s_0=c_+(\lambda^r)\). Since the superdifferential $D^+\rho(s_0)=\emptyset$, \eqref{rho11} is automatically a viscosity subsolution of \eqref{rho1}.
To verify that it is also a supersolution, for any test function \(\varphi\in C^{1}(\mathbb{R})\) such that \(\rho-\varphi\) attains a local minimum at \(s_0\).  Since \(c_1\le c_+(\lambda^r)\), we have \(R^{*}(s_0)=R_{+}\). Hence,
\begin{align*}
&d\left(\int_{\mathbb{R}}J(y)e^{\varphi'(s_{0})y}\,dy-1\right)
+\rho(s_{0})-s_{0}\varphi'(s_{0})-\mu
+\mu R(s_{0})e^{\tau(\rho(s_{0})-s_{0}\varphi'(s_{0}))} \\
=&\Delta_{+}\!\left(\varphi'(s_{0}),\,s_{0}\varphi'(s_{0})\right) \\
=&\Delta_{+}\!\left(\varphi'(s_{0}),\,c_+(\lambda^r)\varphi'(s_{0})\right)
-\Delta_{+}\!\left(\varphi'(s_{0}),\,\mathbf{H}_{+}(\varphi'(s_{0}))\right),
\end{align*}
where we used $\Delta_{+}\!\left(\varphi'(s_{0}),\,\mathbf{H}_{+}(\varphi'(s_{0}))\right)=0$ in the last equality.
Note that $\varphi'(s_0)\in D^-\rho(s_0)=[0,\lambda^r]$, where $D^-\rho(s_0)$ represents the subdifferential of $\rho$ at $s_0$, and the function \(\mathbf{H}\mapsto\Delta_{+}(p,\mathbf{H})\) is strictly decreasing. Observe that $\lambda^r\in(0,\nu_+^*)$. Then, for $\varphi'(s_0)>0$,
\[
\Delta_{+}\!\left(\varphi'(s_{0}),\,c_+(\lambda^r)\varphi'(s_{0})\right)
\ge\Delta_{+}\!\left(\varphi'(s_{0}),\,c_+(\varphi'(s_{0}))\varphi'(s_{0})\right)
=\Delta_{+}\!\left(\varphi'(s_{0}),\mathbf{H}_{+}(\varphi'(s_{0}))\right),
\]
where the inequality follows from \(c_+(\lambda^r)\le c_+(\varphi'(s_{0}))\).  For $\varphi'(s_0)=0$, the inequality follows from $\mathbf{H}_+(0)>0$. Thus $\rho$ is a viscosity supersolution at \(s_{0}\).  Therefore, the comparison principle implies that $\rho$ is the unique viscosity solution of \eqref{rho1}, and the rightmost point where \(\rho\) vanishes is \(s=c_+(\lambda^r)\).  Hence \(c_{r}=c_+(\lambda^r)\).
\end{proof}

\begin{lemma}\label{lem:vis-A2}
If $(\lambda^r,c_1)\in A_2$, then $c_r=c_-(\underline{p}(\lambda^r ,c_1))$.
\end{lemma}
\begin{proof}
Define
\begin{align}\label{rho12}
\rho(s)=
\begin{cases}
\lambda^{r}s-\mathbf{H}_{+}(\lambda^{r}), & s\in[c_1,+\infty),\\
\underline{p}s-\mathbf{H}_{-}(\underline{p}), & s\in[c_-(\underline{p}),c_{1}),\\
0, & s\in[0,c_-(\underline{p})),
\end{cases}
\end{align}
where $\underline{p}:=\underline{p}(\lambda^r ,c_1)$ is defined in Lemma \ref{lem:under-p}. Then \eqref{rho12} is a classical solution of \eqref{rho1} except two points \(s_1=c_{1}\) and \(s_2=c_{-}(\underline{p})\), and satisfies the boundary condition \eqref{rho-boundary-1}. Since $D^+\rho(s_1)=D^+\rho(s_2)=\emptyset$, \eqref{rho12} is automatically a viscosity subsolution of \eqref{rho1}. The verification of supersolution property at $c_{-}(\underline{p})$ is similar to Lemma \ref{lem:vis-A1}, we only show the case $s_1=c_1$. Let \(\varphi\) be a test function such that \(\rho-\varphi\) has a local minimum at \(c_{1}\). For $\varphi'(c_1)\in D^-\rho(c_1)=[\underline{p},\lambda^r]$, it follows that
\begin{align*}
&d\left(\int_{\mathbb{R}}J(y)e^{\varphi'y}dy-1\right) + \rho(c_1)-c_1\varphi'-\mu+\mu R(c_1)e^{\tau(\rho(c_1)-c_1\varphi') }\\
=&\Delta_+(\varphi',c_1\varphi'-\underline{p}c_1+\mathbf{H}_{-}(\underline{p}))-\Delta_+(\varphi',\mathbf{H}_+(\varphi'))\geq0
\end{align*}
where we used $\Delta_+(\varphi',\mathbf{H}_+(\varphi'))=0$ in the first equality, $\mathbf{H}\mapsto\Delta_+(p,\mathbf{H})$ is decreasing in $\mathbf{H}$, and Lemma \ref{lem:under-p} for the last inequality. Hence the supersolution property holds at \(c_{1}\). Consequently \eqref{rho12} is the unique viscosity solution, and the rightmost zero of \(\rho\) occurs at \(s=c_{-}(\underline{p}(c_{1}))\). Therefore \(c_{r}=c_{-}(\underline{p}(c_{1}))\).
\end{proof}

\begin{lemma}\label{lem:vis-A3}
If $(\lambda^r,c_1)\in A_3$, then $c_r=c_-(\bar{p}(c_1))$.
\end{lemma}
\begin{proof}
Define
\begin{align}\label{rho13}
\rho(s)=
\begin{cases}
\lambda^{r}s-\mathbf{H}_{+}(\lambda^{r}), & s\in[\mathbf{H}_{+}'(\lambda^{r}),+\infty),\\
\mathbf{L}_+(s), & s\in[c_{1},\mathbf{H}_{+}'(\lambda^{r})),\\
\bar{p}s-\mathbf{H}_{-}(\bar{p}), & s\in[c_{-}(\bar{p}),c_{1}),\\
0, & s\in[0,c_{-}(\bar{p})).
\end{cases}
\end{align}
The same reasoning as before shows that \eqref{rho13} is the unique viscosity solution of \eqref{rho1}. Thus, $c_r=c_-(\bar{p}(c_1))$.
\end{proof}

\begin{lemma}\label{lem:vis-A4}
If $(\lambda^r,c_1)\in A_4$, then $c_r=c_+^*$.
\end{lemma}
\begin{proof}
Define
\begin{align*}
\rho(s)=
\begin{cases}
\lambda^{r}s-\mathbf{H}_{+}(\lambda^{r}), & s\in[\mathbf{H}_{+}'(\lambda^{r}),+\infty),\\
\mathbf{L}_+(s), & s\in[c_+^*,\mathbf{H}_{+}'(\lambda^{r})),\\
0, & s\in[0,c_+^*).
\end{cases}
\end{align*}
Then, $c_r=c_+^*$.
\end{proof}

\begin{lemma}\label{lem:vis-A5}
If $(\lambda^r,c_1)\in A_5$, then $c_r=c_-^*$.
\end{lemma}
\begin{proof}
We divide it into three cases: (a) $\lambda^r\in(\nu_-^*,\Psi_+(\bar{c}_1))$ and $c_1\ge g(\lambda^r)$, (b) $\lambda^r\in(\Psi_+(\bar{c}_1),\infty)$ and $c_1>\mathbf{H}_+'(\lambda^r)$, (c) $\lambda^r\in(\Psi_+(\bar{c}_1),\infty)$ and $\bar{c}_1\le c_1\le\mathbf{H}_+'(\lambda^r)$.
For cases (a) and (b), define
\begin{align*}
\rho(s)=
\begin{cases}
\lambda^{r}s-\mathbf{H}_{+}(\lambda^{r}), & s\in[c_1,+\infty),\\
\underline{p}s-\mathbf{H}_{-}(\underline{p}), & s\in[\mathbf{H}_-'(\underline{p}),c_{1}),\\
\mathbf{L}_-(s), & s\in[c_-^*,\mathbf{H}_-'(\underline{p})),\\
0, & s\in[0,c_-^*).
\end{cases}
\end{align*}
For case (c), define
\begin{align*}
\rho(s)=
\begin{cases}
\lambda^{r}s-\mathbf{H}_{+}(\lambda^{r}), & s\in[\mathbf{H}_{+}'(\lambda^{r}),+\infty),\\
\mathbf{L}_+(s), & s\in[c_{1},\mathbf{H}_{+}'(\lambda^{r})),\\
\bar{p}s-\mathbf{H}_{-}(\bar{p}), & s\in[c_{-}(\bar{p}),c_{1}),\\
\mathbf{L}_-(s), & s\in[c_-^*,\mathbf{H}_-'(\bar{p})),\\
0, & s\in[0,c_-^*).
\end{cases}
\end{align*}
Therefore, $c_r=c_-^*$.
\end{proof}

\begin{proof}[Proof of Theorem \ref{thm:right}]
Combining Lemmas~\ref{lem:vis-A1}--\ref{lem:vis-A5}, it is easy to check that $c_r(\lambda^r,c_1)$ is a continuous function on $(0,\infty)\times\mathbb{R}$. Then we obtain Theorem~\ref{thm:right}.
\end{proof}

\subsection{Leftward Spreading Speed}
\noindent

In this subsection, adapting the constructions from Section 4 of \cite{Tao2026}, we construct explicit viscosity solutions for the Hamilton-Jacobi equation on the left half-line and derive the leftward spreading speeds.

Set $\tilde{c}_1=-c_1$. Note that $c_+^*>c_-^*$, there exists $\nu_1<\nu_-^*<\nu_2$ such that $c_+^*=c_-(\nu_i)$, $i=1,2$. Denote
\begin{align*}
&E_3 := \{\tilde{c}_1 : \tilde{c}_1 \geq c_+^*\},\\
&E_4 := \{(\lambda^l,\tilde{c}_1) : \tilde{c}_1\lambda^l > \mathbf{H}_-(\lambda^l) \text{ for } \lambda^l \leq \nu_1 \}\cup \{(\lambda^l,\tilde{c}_1):\tilde{c}_1 \geq c_+^* \text{ for } \lambda^l > \nu_1\}.
\end{align*}

\begin{lemma}\label{lem:check-p}
For any \(\tilde{c}_1 \in E_3\), the equation in \(p\)
\[
\tilde{c}_1 p - \mathbf{H}_-(p) = \mathbf{L}_+(\tilde{c}_1)
\]
has two positive roots \(\check{p}:=\check{p}({c}_1)\) and \(\hat{p}:=\hat{p}({c}_1)\) satisfying  \(\check{p} < \hat{p}\). Moreover,
\[
\check{p} < \min\{\Psi_+(\tilde{c}_1),\Psi_-(\tilde{c}_1)\} \le \max\{\Psi_+(\tilde{c}_1),\Psi_-(\tilde{c}_1)\} < \hat{p}.
\]
In particular, \(\check{p}(c_+^*) = \nu_1\) and \(\hat{p}(c_+^*) = \nu_2\), where \(\nu_1<\nu_-^*<\nu_2\) are defined by \(c_+^*=c_-(\nu_i)\) for \(i=1,2\).
\end{lemma}
\begin{proof}
Define
\[
F_3(\tilde{c}_1,p) = \tilde{c}_1 p - \mathbf{H}_-(p) - \mathbf{L}_+(\tilde{c}_1), \quad \tilde{c}_1 \in E_3.
\]
The function \(p \mapsto F_3(\tilde{c}_1,p)\) is increasing on \([0,\Psi_-(\tilde{c}_1)]\) and decreasing on \([\Psi_-(\tilde{c}_1),\infty)\). Observe that
\begin{align*}
&F_3(\tilde{c}_1,0) = -\mathbf{H}_-(0) - \tilde{c}_1\Psi_+(\tilde{c}_1) + \mathbf{H}_+(\Psi_+(\tilde{c}_1)) < 0,\\
&F_3(\tilde{c}_1,\Psi_+(\tilde{c}_1)) = -\mathbf{H}_-(\Psi_+(\tilde{c}_1)) + \mathbf{H}_+(\Psi_+(\tilde{c}_1)) > 0,\quad\lim_{p\to\infty} F_3(\tilde{c}_1,p) = -\infty,
\end{align*}
then \(F_3(\tilde{c}_1,p)=0\) has exactly two positive roots \(\check{p}<\hat{p}\), with \(\check{p}<\Psi_+(\tilde{c}_1)<\hat{p}\) and also \(\check{p}<\Psi_-(\tilde{c}_1)<\hat{p}\). This yields the stated inequalities.

At \(\tilde{c}_1 = c_+^*\), we have \(\Psi_+(c_+^*) = \nu_+^*\) and \(F_3(c_+^*,p) = c_+^* p - \mathbf{H}_-(p)\). The numbers \(\nu_1,\nu_2\) are precisely the two positive roots of \(c_+^* p - \mathbf{H}_-(p)=0\), hence \(\check{p}(c_+^*) = \nu_1\) and \(\hat{p}(c_+^*) = \nu_2\).
\end{proof}

\begin{lemma}\label{lem:p_*}
For any \((\lambda^l,\tilde{c}_1) \in E_4\) satisfying \(\lambda^l < \check{p}({c}_1)\), the equation in \(p\)
\begin{align}\label{eq:p_*}
\tilde{c}_1 p - \mathbf{H}_+(p) = \tilde{c}_1 \lambda^l - \mathbf{H}_-(\lambda^l)
\end{align}
has a unique smallest positive root \(p_*:=p_*(\lambda^l,{c}_1)\). Moreover,
$p_* \in (\lambda^l,\Psi_+(\tilde{c}_1))$,
and \(p_*\) is strictly decreasing in \(\tilde{c}_1\) and strictly increasing in \(\lambda^l\) whenever \(\tilde{c}_1 > \mathbf{H}_-'(\lambda^l)\).
\end{lemma}

\begin{proof}
For $\tilde{c}_1>c_+^*$ and $\nu_1<\lambda^l<\check{p}<\Psi_-(\tilde{c}_1)$, define $h_2(\lambda^l)=\tilde{c}_1\lambda^l-\mathbf{H}_-(\lambda^l)$. Since $\lambda^l<\Psi_-(\tilde{c}_1)$, we have $h_2'(\lambda^l)=\tilde{c}_1-\mathbf{H}_-'(\lambda^l)>0$. Thus
\[
h_2(\lambda^l)\ge h_2(\nu_1)=\tilde{c}_1\nu_1-\mathbf{H}_-(\nu_1)\ge c_+^*\nu_1-\mathbf{H}_-(\nu_1)=0\quad\text{for} \quad\nu_1<\lambda^l<\check{p}.
\]
Hence, $\tilde{c}_1 \lambda^l - \mathbf{H}_-(\lambda^l)\ge0$ in $E_4\cap \{\lambda^l < \check{p}\}$.
Define
\[
F_4(\lambda^l,\tilde{c}_1,p) = \tilde{c}_1 p - \mathbf{H}_+(p) - \tilde{c}_1 \lambda^l + \mathbf{H}_-(\lambda^l), \quad (\lambda^l,\tilde{c}_1)\in E_4\cap \{\lambda^l < \check{p}\}.
\]
For fixed $(\lambda^l,\tilde{c}_1)$, $F_4$ is increasing in $p \in [0, \Psi_+(\tilde{c}_1)]$ and decreasing in $p \in [\Psi_+(\tilde{c}_1), \infty)$.
Using the condition $\lambda^l < \check{p}$ and the definition of $\check{p}$, we have
\[
\tilde{c}_1 \lambda^l - \mathbf{H}_-(\lambda^l) < \tilde{c}_1 \Psi_+(\tilde{c}_1) - \mathbf{H}_+(\Psi_+(\tilde{c}_1)).
\]
Furthermore,
\begin{align*}
&F_4(\lambda^l,\tilde{c}_1, \lambda^l) = -\mathbf{H}_+(\lambda^l) + \mathbf{H}_-(\lambda^l) < 0,  \\
&F_4(\lambda^l,\tilde{c}_1, \Psi_+(\tilde{c}_1)) = \tilde{c}_1\Psi_+(\tilde{c}_1)-\mathbf{H}_+(\Psi_+(\tilde{c}_1))-\tilde{c}_1\lambda^l+\mathbf{H}_-(\lambda^l)>0.
\end{align*}
Consequently, $F_4(\lambda^l,\tilde{c}_1,p)=0$ admits a smallest positive root $p_*(\lambda^l,{c}_1)\in(\lambda^l,\Psi_+(\tilde{c}_1))$. Differentiating \(F_4(\lambda^l,\tilde{c}_1,p)=0\) with respect to \(\tilde{c}_1\) and $\lambda^l$ along \(p=p_*(\lambda^l,{c}_1)\) yields
\[
\frac{\partial F_4}{\partial \tilde{c}_1} + \frac{\partial F_4}{\partial p}\,\frac{\partial p_*}{\partial c_1}=0 \quad {and}\quad\frac{\partial F_4}{\partial \lambda^l} + \frac{\partial F_4}{\partial p}\,\frac{\partial p_*}{\partial \lambda^l}=0.
\]
Note that
\begin{equation}\label{partial-F-4}
\begin{aligned}
&\partial_{\lambda^l} F_4(\lambda^l,\tilde{c},p_*)=-\tilde{c}_1+\mathbf{H}_-'(\lambda^l),\quad\partial_{\tilde{c}_1}F_4(\lambda^l,\tilde{c},p_*)=p_*-\lambda^l>0,\\
&\partial_pF_4(\lambda^l,\tilde{c},p_*)=\tilde{c}_1-\mathbf{H}_+'(p_*)>0.
\end{aligned}
\end{equation}
Hence, $p_*$ is decreasing in $\tilde{c}_1$, and increasing in $\lambda^l$ if $\tilde{c}_1>\mathbf{H}_-'(\lambda^l)$.
\end{proof}

For $(\lambda^l,\tilde{c}_1)\in E_4$ with $\lambda^l \in (0,\nu_+^*)$, consider the equation $F_4(\lambda^l,\tilde{c}_1,\nu_+^*) = 0$.
Note that $F_4(\lambda^l,\tilde{c}_1,\nu_+^*) \to \pm\infty$ as $\tilde{c}_1 \to \pm\infty$ and $F_4$ is strictly increasing in $\tilde{c}_1$. Hence for each fixed $\lambda^l$ there exists a unique $\tilde{c}_1$ satisfying $F_4 = 0$. This defines a function $\tilde{c}_1 = k(\lambda^l)$, which solves $F_4 = 0$ explicitly, and hence
\begin{align}\label{eq:k}
 k(\lambda^l) = \frac{\mathbf{H}_+(\nu_+^*) - \mathbf{H}_-(\lambda^l)}{\nu_+^* - \lambda^l}, \qquad 0 < \lambda^l < \nu_+^*.
\end{align}
\begin{lemma}\label{lem:k(lambda^l)}
Let $k$ be defined by \eqref{eq:k}. Then the following assertions hold true.
\begin{enumerate}
    \item[(i)] $k$ is strictly increasing on $(0,\nu_+^*)$.
    \item[(ii)] $k(\lambda^l) < c_-(\lambda^l)$ for $0 < \lambda^l < \nu_1$, and $k(\lambda^l) > \mathbf{H}_+'(\lambda^l)$ for $\lambda^l \in (\nu_1,\nu_+^*)$.
    \item[(iii)] For any $(\lambda^l,\tilde{c}_1)\in E_4\cap \{\lambda^l < \check{p}\}$ such that $p_*$ is well-defined, we have
    \[
    p_*(\lambda^l,\tilde{c}_1) \leq \nu_+^* \quad\Longleftrightarrow\quad \tilde{c}_1 \geq k(\lambda^l),
    \]
    and the equality holds iff $\tilde{c}_1 = k(\lambda^l)$.
\end{enumerate}
\end{lemma}

\begin{proof}
A direct calculation shows that \(k\) is strictly increasing, and then (i) and (ii) follows immediately. We omit the details. It remains to establish the assertion (iii).
Note that \(\tilde{c}_1\ge c_+^*\) implies \(\nu_+^*\le \Psi_+(\tilde{c}_1)\). It follows from \(p_*\in(\lambda^l,\Psi_+(\tilde{c}_1)]\) that
\begin{align*}
\tilde{c}_1 > k(\lambda^l) \iff& \tilde{c}_1\lambda^l-\mathbf{H}_+(\nu_+^*) > \tilde{c}_1\lambda^l-\mathbf{H}_-(\lambda^l)\\
\iff& \tilde{c}_1\nu_+^*-\mathbf{H}_+(\nu_+^*) > \tilde{c}_1p_*-\mathbf{H}_+(p_*)\\
\iff& p_* < \nu_+^*,
\end{align*}
where the last equivalence follows from the strict monotonicity of the function \(h_3(s)=\tilde{c}_1s-\mathbf{H}_+(s)\) for \(s\in(-\infty,\Psi_+(\tilde{c}_1)]\). The remaining cases are analogous. This completes the proof.
\end{proof}

\begin{lemma}\label{lem:vis-B1}
If $(\lambda^l,c_1)\in B_1$, then $c_l=c_-(\lambda^l)$.
\end{lemma}
\begin{proof}
Define
\begin{align}\label{rho21}
\rho(s)=\max\{0,-\lambda^ls-\mathbf{H}_-(\lambda^l)\}, \quad s \in(-\infty,0].
\end{align}
Note that \eqref{rho21} satisfies the boundary condition \eqref{rho-boundary-2} and is a classical solution of \eqref{rho1} except $s_3:=-c_-(\lambda^l)$. Since $D^+\rho(s_3)=\emptyset$, it is automatically a subsolution. We only need to verify it is also a supersolution of \eqref{rho1}. Let \(\varphi\) be a test function with \(\rho-\varphi\) having a local minimum at \(s_3\).
For $\varphi'(s_3)\in D^-\rho(s_3)=[-\lambda^l,0]$, we have
\begin{align*}
&d\left(\int_{\mathbb{R}}J(y)e^{\varphi'y}\,dy-1\right)+\rho(s_3)
-s_3\varphi'-\mu
+\mu R(s_3)e^{\tau(\rho(s_3)-s_3\varphi')} \\
=&\Delta_{-}\left(\varphi',\,-c_-(\lambda^l)\varphi'\right) - \Delta_-(\varphi',\mathbf{H}_-(\varphi'))\ge0,
\end{align*}
where we used $\Delta_-(p,\mathbf{H})$ is decreasing in $\mathbf{H}$, $c_-(p)$ is increasing in $(0,\nu_-^*)$ and $\lambda^l\le\nu_-^*$. Then, $c_l=c_-(\lambda^l)$.
\end{proof}

\begin{lemma}\label{lem:vis-B2}
If $(\lambda^l,c_1)\in B_2$, then $c_l=c_+(p_*(\lambda^l ,c_1))$.
\end{lemma}
\begin{proof}
Note that $-c_1\geq k(\lambda^l)$, and $\lambda^l<p_*\leq \nu_+^*$. Define
\begin{align}\label{rho22}
\rho(s)=
\begin{cases}
-\lambda^ls-\mathbf{H}_-(\lambda^l), & s\in (-\infty,c_1],\\
-p_*s-\mathbf{H}_+(p_*), & s\in(c_1,-c_+(p_*)],\\
0, & s\in(-c_+(p_*),0].
\end{cases}
\end{align}
Note that \eqref{rho22} satisfies the boundary condition \eqref{rho-boundary-2} and is a classical solution of \eqref{rho1} except $\{c_1,-c_+(p_*)\}$. Since $D^+\rho(c_1)=\emptyset$, we need to verify subsolution property at $s_4:=-c_+(p_*)$.
Let \(\varphi\) be a test function with \(\rho-\varphi\) having a local maximum at \(s_4\). For $\varphi'(s_4)\in D^+\rho(s_4)=[-p_*,-\lambda^l]$, we have
\begin{align*}
&d\left(\int_{\mathbb{R}}J(y)e^{\varphi'y}\,dy-1\right)+\rho(s_4)
-s_4\varphi'-\mu
+\mu R_*(s_4)e^{\tau(\rho(s_4)-s_4\varphi')} \\
=&\Delta_{-}\left(\varphi',\,s_4\varphi'-\rho(s_4)\right) - \Delta_-(\varphi',\mathbf{H}_-(\varphi'))\le0,
\end{align*}
where we used $\Delta_-(p,\mathbf{H})$ is decreasing in $\mathbf{H}$ and Lemma \ref{lem:p_*}. Hence, \eqref{rho22} is a subsolution. Since $D^-\rho(c_1)=\emptyset$, \eqref{rho22} automatically satisfies the supersolution property at $c_1$. The remaining verification for supersolution at $-c_+(p_*)$ is similar to Lemma \ref{lem:vis-B1} and we omit it here. Then, $c_l=c_+(p_*)$.
\end{proof}

\begin{lemma}\label{lem:vis-B3}
If $(\lambda^l,c_1)\in B_3$, then $c_l=-c_1$.
\end{lemma}
\begin{proof}
Since $c_1\le-c_-^*$, there exist two constants $\underline{\nu},\bar{\nu}$ such that
\[
\nu_1<\underline{\nu}\le\nu_-^*\le\bar{\nu}\quad \text{ and }\quad -c_1=c_-(\bar{\nu})=c_-(\underline{\nu}).
\]
Define
\begin{align}\label{eq:B3-sub}
\underline\rho(s)=\max\{0,-\underline{\nu}s-\mathbf{H}_-(\underline{\nu})\}, \quad s\in(-\infty,0].
\end{align}
Then \eqref{eq:B3-sub} is automatically a subsolution of \eqref{rho1}.
On the other hand, define
\begin{align}\label{eq:B3-sup}
\bar\rho(s)=
\begin{cases}
\mathbf{L}_-(s), & s\in(-\infty,-\mathbf{H}_-'(\bar{\nu})],\\
-\bar{\nu} s-\mathbf{H}_-(\bar{\nu}), & s\in(-\mathbf{H}_-'(\bar{\nu}), c_1],\\
0, & s\in(c_1,0].
\end{cases}
\end{align}
Then \eqref{eq:B3-sup} is a supersolution of \eqref{rho1}. Then the comparison principle implies that the viscosity solution $\rho$ of \eqref{rho1} satisfies $\rho(s)=0$ iff $s\in[c_{1},0]$. Hence, $c_l=-c_1$.
\end{proof}

\begin{lemma}\label{vis-B4}
If $(\lambda^l,c_1)\in B_4$, then $c_l=c_-^*$.
\end{lemma}
\begin{proof}
Define
\begin{align}\label{rho25}
\rho(s)=
\begin{cases}
-\lambda^ls-\mathbf{H}_-(\lambda^l), & s\in(-\infty,-\mathbf{H}_-'(\lambda^l)],\\
s\Psi_-(s)-\mathbf{H}_-(\Psi_-(s)), & s\in(-\mathbf{H}_-'(\lambda^l),-c_-^*],\\
0, & s\in(-c_-^*,0],
\end{cases}
\end{align}
Then \eqref{rho25} is a classical solution of \eqref{rho1} at every point except \(-c_-^*\), and satisfies the boundary condition \eqref{rho-boundary-2}. Note that $D^+\rho(-c_-^*)=\emptyset$, \eqref{rho25} is a viscosity subsolution automatically. We only need to verify that \eqref{rho25} is a viscosity supersolution at $-c_-^*$. Let \(\varphi\) be a test function with \(\rho-\varphi\) having a local minimum at \(-c_-^*\).
For $\varphi'(-c_-^*)\in D^-\rho(-c_-^*)=[\Psi_-(-c_-^*),0]$, we have
\begin{align*}
&d\left(\int_{\mathbb{R}}J(y)e^{\varphi'y}\,dy-1\right)+\rho(-c_-^*)
+c_-^*\varphi'-\mu
+\mu R(-c_-^*)e^{\tau(\rho(-c_-^*)+c_-^*\varphi')} \\
=&\Delta_{-}\!\left(\varphi',\,-c_-^*\varphi'\right) - \Delta_-(\varphi',\mathbf{H}_-(\varphi'))\\
=&\Delta_{-}\!\left(\varphi',\,-c_-(\nu_-^*)\varphi'\right) - \Delta_-(\varphi',\mathbf{H}_-(\varphi'))
\geq0,
\end{align*}
where we used $\Delta_-(p,\mathbf{H})$ is decreasing in $\mathbf{H}$, $c_-(p)$ is increasing in $[\Psi_-(-c_-^*),0]$ and $\Psi_-(-c_-^*)=-\nu_-^*$. Then, $c_l=c_-^*$.
\end{proof}

\begin{lemma}\label{lem:vis-B5}
If $(\lambda^l,c_1)\in B_5$, then $c_l=c_+^*$.
\end{lemma}
\begin{proof}
For fixed $c_{1}$, there exists a unique $\lambda_{k}\in(0,\nu_{+}^*)$ such that $-c_{1}=k(\lambda_{k})$.
Define
\begin{align}\label{eq:B5-sub}
\underline\rho(s)=
\begin{cases}
-\lambda_{k}s-\mathbf{H}_-(\lambda_{k}), & s\in(-\infty,c_{1}],\\
\max\left\{0,-p_{*}s-\mathbf{H}_+(p_{*})\right\}, & s\in(c_1,0].
\end{cases}
\end{align}
By Lemma \ref{lem:k(lambda^l)}, we have $p_{*}=\nu_{+}^*$.
Notably, \eqref{eq:B5-sub} is automatically a subsolution of \eqref{rho1}.
We next define
\begin{align}\label{eq:B5-sup}
\rho(s)=
\begin{cases}
s\Psi_-(s)-\mathbf{H}_-(\Psi_-(s)), & s\in(-\infty,-\mathbf{H}_{-}'(\hat{p})],\\
-\hat{p}s-\mathbf{H}_-(\hat{p}), & s\in(-\mathbf{H}_{-}'(\hat{p}), c_1],\\
s\Psi_+(s)-\mathbf{H}_+(\Psi_+(s)), & s\in(c_1,-c_+^*],\\
0, & s\in(-c_+^*,0].
\end{cases}
\end{align}
By a similar argument, \eqref{eq:B5-sup} is a supersolution of \eqref{rho1}. Then $c_l = c_+^*$.
\end{proof}

\begin{proof}[Proof of Theorem \ref{thm:left}]
 Combining Lemmas~\ref{lem:vis-B1}--\ref{lem:vis-B5}, one can readily verify that the leftward spreading speed $c_l(\lambda^l,c_1)$ is continuous on $(0,\infty)\times\mathbb{R}$. This immediately yields Theorem~\ref{thm:left}.
\end{proof}

\begin{proof}[Proof of Theorem \ref{thm:vis-infty}]
We only show the rightward spreading speeds, since the leftward one follows from a similar fashion. Here, we only consider the case $c_+^*<c_1<\bar{c}_1$, the others are similar. By Lemma \ref{lem:bar-p}, $\bar{p}(c_1)<\nu_-^*$ is well defined. Letting $\lambda^r\to\infty$ in \eqref{rho13}, then \eqref{rho13} converges locally to $\tilde{\rho}$, where
\[
\tilde{\rho}(s)=
\begin{cases}
\mathbf{L}_+(s), & s\in[ c_{1},+\infty),\\
\bar{p}s-\mathbf{H}_{-}(\bar{p}), & s\in[c_{-}(\bar{p}),c_{1}),\\
0, & s\in[0,c_{-}(\bar{p})).
\end{cases}
\]
By stability property of viscosity solutions \cite[Theorem 6.2]{Barles2013}, $\tilde{\rho}$ is a viscosity solution of \eqref{rho1} in $(0,\infty)$. Since $\tilde{\rho}(0)=0$, and
\[
\lim_{s \to \infty} \frac{\tilde{\rho}(s)}{s} \geq \lim_{s \to \infty} \frac{\rho(s)}{s} = \lambda^r \quad \text{for each } \lambda^r \in [\nu_+^*, \infty).
\]
Letting $\lambda^r\to\infty$, then \eqref{rho-boundary-1} holds. By the uniqueness result of viscosity solutions, we know that $\tilde{\rho}$ gives the unique viscosity solution of \eqref{rho1}, and thus, $c_r=c_-(\bar{p}(c_1))$.
\end{proof}

\section{Effects of time delay on spreading speeds}\label{Sec:tau}
\noindent

Next, we discuss the effect of time delay on the spreading speed. Let \(\mathbf{H}_\pm(p;\tau)\) denote the unique solutions of \eqref{gamma}, viewed as a function of both $p$ and $\tau$. All other quantities derived from \(\mathbf{H}_\pm\) (such as \(\nu_\pm^*\), \(\Psi_\pm\), etc.) are understood to depend on \(\tau\) as well, though we do not always write this dependence explicitly when it is clear from context.

\begin{lemma}\label{lem:monoton-not-strict}
Assume that {\rm (J), (F1)--(F4)} and {\rm(\text{IC}$^\lambda)$} hold with some $\lambda\in(0,\infty]$.
The rightward spreading speed \(c_r(\lambda^r,c_1;\tau)\) is strictly decreasing in \(\tau\ge 0\), while the leftward spreading speed \(c_l(\lambda^l,c_1;\tau)\) is non-increasing in \(\tau\ge 0\).
\end{lemma}
\begin{proof}
This proof is divided into several steps.

\noindent\textbf{Step 1.} For any \(p>0\), the maps \(\tau\mapsto \mathbf{H}_+(p;\tau)/p\) and \(\tau\mapsto \mathbf{H}_-(p;\tau)/p\) are strictly decreasing.

We omit the subscript \(\pm\) throughout this step for notational simplicity.
Consider \(\mathbf{H}(p;\tau)>0\) defined implicitly by
\[
\Delta(p,\mathbf{H},\tau):=d\left(\int_{\mathbb{R}}J(y)e^{py}\,dy-1\right)-\mathbf{H}-\mu+\mu R e^{-\tau\mathbf{H}}=0.
\]
By the implicit function theorem, \(\mathbf{H}\) is differentiable with respect to \(\tau\). Differentiating the identity \(\Delta(p,\mathbf{H}(p;\tau),\tau)=0\) with respect to \(\tau\) yields
\[
\frac{\partial\Delta}{\partial\mathbf{H}}\frac{\partial \mathbf{H}}{\partial \tau}+\frac{\partial\Delta}{\partial\tau}=0.
\]
Denote \(X(\tau):=\mu R e^{-\tau\mathbf{H}}>0\).
Direct computation gives $\partial_\mathbf{H}\Delta = -1-\tau X(\tau)$, $\partial_\tau\Delta = -\mathbf{H} X(\tau)$.
Then we obtain
\begin{align}\label{eq:par-gamma-tau}
\frac{\partial \mathbf{H}}{\partial \tau}= -\frac{\mathbf{H} X(\tau)}{1+\tau X(\tau)}<0,
\end{align}
Thus, \(\mathbf{H}(p;\tau)\) is strictly decreasing in \(\tau\ge0\).

\noindent\textbf{Step 2.} For \((\lambda^r,c_1)\in A_2\), $c_r(\lambda^r,c_1;\tau)=\frac{\mathbf{H}_-(\underline{p};\tau)}{\underline{p}}$ is strictly decreasing in \(\tau\).

Define \(L_1(\tau):=c_1\lambda^r-\mathbf{H}_+(\lambda^r;\tau)\) for $(\lambda^r,c_1)\in A_2$. By Step 1, \(L_1(\tau)\) is strictly increasing. Differentiating \eqref{eq:under-p} along $p=\underline{p}$ with respect to \(\tau\) yields
\[
c_1\frac{\partial\underline{p}}{\partial\tau}-\frac{\partial \mathbf{H}_-}{\partial p}\frac{\partial\underline{p}}{\partial\tau}-\frac{\partial \mathbf{H}_-}{\partial\tau}=L_1'(\tau),
\]
Hence
\begin{align}\label{eq:par-under-p}
\frac{\partial\underline{p}}{\partial\tau}=\frac{L_1'(\tau)+\partial_\tau\mathbf{H}_-(\underline{p};\tau)}{c_1-\partial_p\mathbf{H}_-(\underline{p};\tau)}.
\end{align}
Set \(G_1(\tau):=\frac{\mathbf{H}_-(\underline{p};\tau)}{\underline{p}}=c_1-\frac{L_1(\tau)}{\underline{p}}\). Differentiating and using \eqref{eq:par-under-p} yields
\begin{align*}
G'(\tau)&=-\frac{1}{\underline{p}^2}\left(L_1'(\tau)\underline{p}-L_1(\tau)\frac{\partial\underline{p}}{\partial\tau}\right)\\
&=-\frac{1}{\underline{p}^2\left(c_1-\partial_p\mathbf{H}_-(\underline{p};\tau)\right)}\left(L_1'(\tau)\left(\mathbf{H}_-(\underline{p};\tau)-\underline{p}\partial_p\mathbf{H}_-(\underline{p};\tau)\right)-L_1(\tau)\partial_\tau\mathbf{H}_-(\underline{p};\tau)\right).
\end{align*}
Since \((\lambda^r,c_1)\in A_2\), Lemmas \ref{lem:under-p} and \ref{lem:g(lambda^r)} imply \(\underline{p}<\min\{\nu_-^*,\lambda^r\}\). This together with the convexity of \(\mathbf{H}_-\) gives \(\mathbf{H}_-(\underline{p};\tau)-\underline{p}\partial_p\mathbf{H}_-(\underline{p};\tau)>0\). Moreover, since $\underline{p}<\Psi_-(c_1)$, we have $\partial_p\mathbf{H}_-(\underline{p};\tau)<\partial_p\mathbf{H}_-(\Psi_-(c_1);\tau)=c_1$. Consequently, \(G_1'(\tau)<0\).

Both monotonicity properties in \(A_3\) and \(B_2\) can be derived via the same analysis and computations as in Step 2, with only minor adjustments to the auxiliary functions and convexity arguments. We omit it here.
By continuity of the speeds across the demarcation lines, we conclude that \(c_r\) is strictly decreasing for all \((\lambda^r,c_1)\). However, in region \(B_3\) the leftward spreading speed equals \(-c_1\), which is independent of \(\tau\). Consequently, the leftward spreading speed \(c_l\) is decreasing but not strictly in \(\tau\ge 0\). This ends the proof.
\end{proof}

\begin{lemma}\label{lem:nu-decreasing}
$\nu_\pm^*(\tau)$ is strictly decreasing in $\tau\ge0$.
\end{lemma}
\begin{proof}
We omit the subscript \(\pm\) throughout this proof for notational simplicity.
Recall that $\nu^*(\tau)$ is the unique positive minimizer of
\[
c(p;\tau):=\frac{\mathbf{H}(p;\tau)}{p},\quad p>0,
\]
which is equivalently characterized by the first-order optimality condition
\begin{equation*}\label{eq:par-c-nu_-}
\partial_p \mathbf{H}(\nu^*(\tau);\tau)
=\frac{\mathbf{H}(\nu^*(\tau);\tau)}{\nu^*(\tau)}
=c^*(\tau),\qquad
\partial_p c(\nu^*(\tau);\tau)=0.
\end{equation*}
Differentiating the above relation with respect to $\tau$ yields
\begin{equation}\label{eq:nuprime}
\frac{{\rm d}\nu^*}{{\rm d}\tau}
= -\,\frac{\partial_\tau\partial_p c(\nu^*(\tau),\tau)}
{\partial_{pp}c(\nu^*(\tau),\tau)}.
\end{equation}
At the minimizer $p=\nu^*(\tau)$, direct computation gives
\begin{align}\label{eq:par-c-pp}
\partial_{pp}c(\nu^*(\tau);\tau)
=\frac{1}{p^3}\,\mathbf{H}_{pp}(\nu^*(\tau);\tau)\,p^2>0,
\end{align}
so the sign of $\mathrm{d}\nu^*/\mathrm{d}\tau$ is opposite to that of the mixed partial derivative $\partial_\tau\partial_p c$.

We next compute $\partial_\tau c(p;\tau)$.
By \eqref{eq:par-gamma-tau}, we obtain
\begin{align*}
\partial_\tau c(p;\tau)
=\frac{1}{p}\partial_\tau\mathbf{H}(p;\tau)
= -\frac{\mathbf{H}}{p}\cdot
\frac{X(\tau)}{1+\tau X(\tau)}
= -\frac{cX(\tau)}{1+\tau X(\tau)}.
\end{align*}
Differentiating with respect to $p$, we obtain
\begin{align}\label{eq:par-p-tau-c}
\partial_p(\partial_\tau c)
= -
\frac{
\big(1+\tau X(\tau)\big)\partial_p\big(c X(\tau)\big)
- c X(\tau)\,\partial_p\big(1+\tau X(\tau)\big)
}
{\big(1+\tau X(\tau)\big)^2}.
\end{align}
Evaluating at $p=\nu^*(\tau)$, we use $\partial_p c=0$ and $\partial_p\mathbf{H} = c^*$ to simplify
\begin{align*}
&\partial_p\big(c X(\tau)\big)
= -(c^*)^2 \tau X(\tau),\quad\partial_p\big(1+\tau X(\tau)\big)
= - c^*\tau^2 X(\tau).
\end{align*}
Substituting these into \eqref{eq:par-p-tau-c} at $p=\nu^*$, we get
\begin{align*}
\partial_p\partial_\tau c(\nu_*(\tau);\tau)
&= -\frac{1}{\big(1+\tau X(\tau)\big)^2}
\left[
-(c^*)^2\tau X(\tau)\big(1+\tau X(\tau)\big)
+(c^*)^2\tau^2 X^2(\tau)
\right]\\
&= \frac{ (c^*)^2\tau X(\tau)}
{\big(1+\tau X(\tau)\big)^2} > 0.
\end{align*}
Hence $\partial_p\partial_\tau c(\nu_*(\tau);\tau)>0$.
Combined with \eqref{eq:nuprime} and \eqref{eq:par-c-pp}, we conclude $\frac{{\rm d}\nu^*}{{\rm d}\tau}<0$.
Therefore, $\nu^*(\tau)$ is strictly decreasing for $\tau\ge0$.
\end{proof}

\begin{lemma}
\(c_\pm(\tau)\) and \(c_\pm^*(\tau)\) are strictly convex with respect to \(\tau\ge0\).
\end{lemma}

\begin{proof}
We omit the subscript \(\pm\) throughout this proof for notational simplicity.
For fixed \(p>0\), the function \(\mathbf{H}(p;\tau)\) is implicitly defined by
\[
d\int_{\mathbb{R}}J(y)e^{py}\,dy-d-\mu-\mathbf{H}(p;\tau)+\mu R e^{-\tau\mathbf{H}(p;\tau)}=0.
\]
Define the constant $B:=\mu -d\int_{\mathbb{R}}J(y)e^{py}\,dy+d$,
so the above equation reduces to
\begin{align}\label{eq:gamma}
\mathbf{H}+B=\mu R e^{-\tau\mathbf{H}}.
\end{align}
We now regard \(\tau\) as a smooth function of \(\mathbf{H}\) for fixed \(p\). Solving for \(\tau\) yields
\[
\tau(\mathbf{H})=\frac{1}{\mathbf{H}}\ln\frac{\mu R}{\mathbf{H}+B},
\]
which is well-defined since \(\mathbf{H}+B>0\).
Differentiating \eqref{eq:gamma} with respect to \(\mathbf{H}\), we obtain
\begin{align}\label{eq:1}
1 = -\mu R e^{-\tau\mathbf{H}}\big(\tau'\mathbf{H}+\tau\big)
= -\big(\mathbf{H}+B\big)\big(\tau'\mathbf{H}+\tau\big).
\end{align}
Rearranging gives
\begin{align}\label{eq:tau'}
\tau'(\mathbf{H}) = -\frac{1}{\mathbf{H}(\mathbf{H}+B)}-\frac{\tau}{\mathbf{H}}<0,
\end{align}
which implies that \(\tau(\mathbf{H})\) is strictly decreasing.
Differentiating \eqref{eq:1} once more,
\[
0 = -\big(\tau'\mathbf{H}+\tau\big) - \big(\mathbf{H}+B\big)\big(\tau''\mathbf{H}+2\tau'\big).
\]
Substituting \eqref{eq:tau'} and simplifying, we arrive at
\[
\tau'' = \frac{1}{\mathbf{H}}\left(\frac{1}{(\mathbf{H}+B)^2}-2\tau'\right) >0,
\]
where the positivity follows from \(\tau'(\mathbf{H})<0\). Hence \(\tau(\mathbf{H})\) is strictly convex.

By the inverse function theorem, the inverse function \(\mathbf{H}(\tau)\) satisfies 
\[
\mathbf{H}''(\tau) = -\frac{\tau''(\mathbf{H})}{\big(\tau'(\mathbf{H})\big)^3}>0,
\]
which shows that \(\tau\mapsto \mathbf{H}(p;\tau)\) is strictly convex for each fixed \(p>0\).
As a direct consequence, \(c(\tau)=\mathbf{H}(p;\tau)/p\) is also strictly convex in \(\tau\ge0\).

It remains to verify the strict convexity of the minimal speed \(c^*(\tau)\).
Recall that
\begin{align}\label{eq:c^*}
c^*(\tau)=\partial_p\mathbf{H}\big(\nu^*(\tau);\tau\big)
=\frac{\mathbf{H}\big(\nu^*(\tau);\tau\big)}{\nu^*(\tau)}.
\end{align}
Differentiating \eqref{eq:c^*} with respect to \(\tau\) and using the optimality condition at \(p=\nu^*(\tau)\), we simplify
\[
\frac{{\rm d}c^*(\tau)}{{\rm d}\tau}
=\frac{\big((\nu^*)'\partial_p\mathbf{H}+\partial_\tau\mathbf{H}\big)\nu^*-(\nu^*)'\mathbf{H}}{(\nu^*)^2}
=\frac{\partial_\tau\mathbf{H}}{\nu^*}.
\]
Combining with \eqref{eq:par-gamma-tau}, we write $\frac{{\rm d}c^*(\tau)}{{\rm d}\tau}
=-c^*(\tau)\,\frac{X(\tau)}{1+\tau X(\tau)}$.
Differentiating once again, a direct computation yields
\begin{align*}
\frac{{\rm d^2}c^*(\tau)}{{\rm d}\tau^2}
=&-\frac{{\rm d}c^*(\tau)}{{\rm d}\tau}\cdot\frac{X(\tau)}{1+\tau X(\tau)}
-c^*(\tau)\cdot\frac{X'(\tau)(1+\tau X(\tau))-X(\tau)\big(X(\tau)+\tau X'(\tau)\big)}{(1+\tau X(\tau))^2}\\
=&c^*(\tau)\cdot\frac{2X^2(\tau)-X'(\tau)}{\big(1+\tau X(\tau)\big)^2}.
\end{align*}
Moreover, one has \(X'(\tau)=X(\tau)\big(-\mathbf{H}+\tau\partial_\tau\mathbf{H}\big)<0\) by \eqref{eq:par-gamma-tau}.
Therefore, \(2X^2(\tau)-X'(\tau)>0\), which implies $\frac{{\rm d^2}c^*(\tau)}{{\rm d}\tau^2}>0$.
This completes the proof.
\end{proof}

\begin{lemma}\label{lem:c_infty}
$\lim\limits_{\tau\to+\infty}c_\pm^*(\tau)=0$.
\end{lemma}
\begin{proof}
We omit the subscript \(\pm\) throughout this proof for notational simplicity.
Note that for each fixed $p>0$, the map $\tau\mapsto c(p;\tau)=\frac{\mathbf{H}(p;\tau)}{p}$ is strictly decreasing in $\tau\ge0$, and Lemma \ref{lem:gamma-pm} gives $c^*({\tau})>0$ for all $\tau\ge0$.
Thus $c^*(+\infty)$ exists and is nonnegative.

Suppose on the contrary that $c^*(+\infty)=c_0>0$. By the definition of infimum, this implies $\lim\limits_{\tau\to+\infty}\mathbf{H}(p;\tau)\ge c_0 p$ for all $p>0$ and all $\tau\ge0$.
For any fixed $p>0$, we have $\tau\mathbf{H}(p;\tau)\ge \tau c_0 p\to+\infty$ as $\tau\to+\infty$, hence $e^{-\tau\mathbf{H}(p;\tau)}\to0$.
Substituting into the implicit equation for $\mathbf{H}(p;\tau)$ and taking the limit as $\tau\to+\infty$, we obtain
\[
c_0 p \le \lim_{\tau\to+\infty}\mathbf{H}(p;\tau) = d\left(\int_{\mathbb{R}}J(y)e^{py}\,dy-1\right)-\mu.
\]
Let $M(p)=d\int_{\mathbb{R}}J(y)e^{py}dy-d$. By continuity of $M(p)$ and $M(0)=0$, we have 
\[
\lim\limits_{p\to0^+}M(p)-\mu = -\mu<0.
\]
Thus there exists a sufficiently small $p_0>0$ such that $M(p_0)-\mu < 0$,
which gives $c_0 p_0 < 0$, contradicting $c_0>0$ and $p_0>0$.
Therefore $c^*(+\infty)=0$.
\end{proof}

According to Theorem~\ref{thm:left}, for a fixed $\tau$, the leftward spreading speed equals $-c_1$ exactly when $(\lambda^l,c_1)$ lies in region $B_3(\tau)$.
Reading the definition of $B_3$ in terms of the critical quantities and using $\tilde c_1=-c_1>0$, we have
\begin{align}\label{eq:B3}
B_3(\tau)=\Bigl\{(\lambda^l,\tilde{c}_1):\
&\bigl[\lambda^l\in[\nu_1(\tau),\nu_-^*(\tau))\ \text{and}\ c_+^*(\tau)\ge\tilde c_1\ge c_-(\lambda^l;\tau)\bigr]\notag\\
\text{or}\quad
&\bigl[\lambda^l\in[\nu_-^*(\tau),+\infty)\ \text{and}\ c_+^*(\tau)\ge\tilde c_1\ge c_-^*(\tau)\bigr]\Bigr\}.
\end{align}
If $\tilde{c}_1 \ge c_+^*(0)$, the strict decrease of $c_+^*(\tau)$ forces $\tilde{c}_1 \ge c_+^*(\tau)$ for all $\tau\ge0$.
If $\tilde{c}_1 \le 0$, then $\tilde{c}_1 < c_-^*(\tau)$ for all $\tau\ge0$.
In either situation the condition \eqref{eq:B3} is never fulfilled, so the locking phenomenon does not occur.

We now work in the admissible range $\tilde{c}_1 \in \bigl(0,\,c_+^*(0)\bigr)$,
and examine the structure of the set $B_3(\tau)$.
Define
\begin{equation*}
\begin{aligned}
&\tau_+ := \inf\,\{\tau\ge 0 : c_+^*(\tau) \le \tilde{c}_1\}, &\tau_- := \inf\,\{\tau\ge 0 : c_-^*(\tau) \le \tilde{c}_1\},\\
&\tau_\lambda := \inf\,\{\tau\ge 0 : c_-(\lambda^l;\tau) \le \tilde{c}_1\}, &\hat\tau := \inf\,\{\tau\ge0 : \nu_-^*(\tau) \le \lambda^l\},
\end{aligned}
\end{equation*}
with the notation that $\inf \emptyset = +\infty$.
By virtue of the strict monotonicity and continuity of $c_\pm^*(\cdot),c_-(\lambda;\tau),\nu_-^*(\tau)$ on $[0,\infty)$ combined with Lemma \ref{lem:c_infty}, it follows that $\tau_+\in(0,+\infty)$ and $\tau_-\in[0,+\infty)$, whereas $\tau_\lambda,\hat{\tau}\in[0,+\infty]$. More precisely, $\tau_-=0$ holds whenever $c_1\in(-c_+^*(0),-c_-^*(0)]$; $\tau_\lambda=+\infty$ if and only if $\tilde{c}_1\le c_-(\lambda^l;+\infty)$; and $\hat\tau=+\infty$ is equivalent to $\lambda^l\le\nu_-^*(+\infty)$.

\begin{lemma}\label{lem:tau+tau-}
Let $\tau_\pm,\tau_\lambda,\hat\tau$ be defined above. We have the following results.
\begin{enumerate}
    \item[(i)] $\tau_- < \tau_+$.
    \item[(ii)] $0 \le \tau_- \le \tau_\lambda$.
    \item[(iii)] If $0<\hat\tau < +\infty$, $\hat\tau < \tau_\lambda$ implies $\hat\tau < \tau_-$.
\end{enumerate}
\end{lemma}

\begin{proof}
(i) From $c_+^*(\tau) > c_-^*(\tau)$ for all $\tau\ge0$ we obtain $\tau_-<\tau_+$ immediately.

(ii) If $\tau_\lambda=0$, then $c_-(\lambda^l;0)\le\tilde c_1$. Since
$c_-(\lambda^l;\tau)\ge c_-^*(\tau)$, we have $c_-^*(0)\le\tilde c_1$,
hence $\tau_-=0$ and $\tau_-=\tau_\lambda$.
If $0<\tau_\lambda<+\infty$, continuity gives $c_-(\lambda^l;\tau_\lambda)=\tilde c_1$,
and $c_-(\lambda^l;\tau_\lambda)\ge c_-^*(\tau_\lambda)$ implies $\tilde c_1\ge c_-^*(\tau_\lambda)$,
so $\tau_-\le\tau_\lambda$ by definition.
If $\tau_\lambda=+\infty$, the inequality $\tau_-\le\tau_\lambda$ is trivial.
Thus $0\le\tau_-\le\tau_\lambda$ in all cases.

(iii) By continuity $\nu_-^*(\hat\tau) = \lambda^l$, hence $c_-(\lambda^l;\hat\tau) = c_-^*(\hat\tau)$ by definition of the minimizer $\nu_-^*(\tau)$. By the strict monotonic decrease of $c_-(\lambda^l;\cdot)$ and $c_-^*(\cdot)$, we derive that
\[
\hat\tau < \tau_\lambda
\implies c_-(\lambda^l;\hat\tau) > c_-(\lambda^l;\tau_\lambda)
\implies c_-^*(\hat\tau) > \tilde{c}_1
\implies \hat\tau < \tau_-.
\]
This chain remains valid even when $\tau_\lambda = +\infty$, because in that case $c_-(\lambda^l;\tau) > \tilde{c}_1$ for all $\tau\ge0$. This ends the proof.
\end{proof}

\begin{lemma}\label{lem:monotonicity}
Assume that $c_1\in(-c_+^*(0),0)$.
Then the leftward spreading speed $c_l(\lambda^l,c_1;\tau)$ fails to be strictly monotone in $\tau$ if and only if $\tau$ belongs to
\( \mathcal{I}_{\rm lock}:=
\begin{cases}
[\tau_\lambda,\tau_+], & \text{if } \hat{\tau}\ge\tau_\lambda,\\[6pt]
[\tau_-,\tau_+],   & \text{if } \hat{\tau}<\tau_\lambda .
\end{cases}
\)
\end{lemma}

\begin{proof}
The strict monotonicity yields the following equivalences, which hold for all $\tau\in[0,\infty)$ even when the thresholds are $+\infty$:
\begin{equation}\label{eq:tau-equiv}
\begin{aligned}
&c_+^*(\tau) \ge \tilde{c}_1 \iff \tau \le \tau_+, \quad
c_-^*(\tau) \ge \tilde{c}_1 \iff \tau \le \tau_-, \quad
c_-(\lambda^l;\tau) \le \tilde{c}_1 \iff \tau \ge \tau_\lambda,\\
&\lambda^l < \nu_-^*(\tau) \;\Longleftrightarrow\; \tau < \hat\tau,\quad
\lambda^l \ge \nu_-^*(\tau) \;\Longleftrightarrow\; \tau \ge \hat\tau.
\end{aligned}
\end{equation}
Consider the first alternative of $B_3$.
Take any $\tau$ satisfying $\tau_\lambda\le\tau\le\tau_+$ and $\lambda^l<\nu_-^*(\tau)$, then
$c_+^*(\tau)\ge\tilde c_1\ge c_-(\lambda^l;\tau) $.
In particular $c_+^*(\tau)\ge c_-(\lambda^l;\tau)$.
Recall that $\nu_1(\tau)$ is the smallest positive root of $c_+^*(\tau)=\mathbf{H}_-(\nu;\tau)/\nu$.
The function $\nu\mapsto\mathbf{H}_-(\nu;\tau)/\nu$ decreases strictly on $(0,\nu_-^*(\tau))$.
If $\lambda^l<\nu_1(\tau)$, then
\[
c_-(\lambda^l;\tau)=\frac{\mathbf{H}_-(\lambda^l;\tau)}{\lambda^l}
>\frac{\mathbf{H}_-(\nu_1(\tau);\tau)}{\nu_1(\tau)}=c_+^*(\tau),
\]
contradicting $c_-(\lambda^l;\tau)\le c_+^*(\tau)$. Hence $\lambda^l\ge\nu_1(\tau)$ is automatically true.
Thus, under the sole extra condition $\lambda^l<\nu_-^*(\tau)$, the first alternative of $B_3$ reduces to $\tau_\lambda\le\tau\le\tau_+$.
The second alternative of $B_3$ translates directly to $\tau_-\le\tau\le\tau_+$ and $\lambda^l\ge\nu_-^*(\tau)$.
Therefore, the set $S$ of delays for which $c_l(\tau)\equiv -c_1$ is
\begin{equation}\label{eq:S}
S = \bigl\{[\tau_\lambda,\tau_+]\cap S_1\bigr\}\;\cup\;\bigl\{[\tau_-,\tau_+]\cap S_2\bigr\},
\end{equation}
with $S_1=\{\tau\ge0:\lambda^l<\nu_-^*(\tau)\}$, $S_2=\{\tau\ge0:\lambda^l\ge\nu_-^*(\tau)\}$.
From \eqref{eq:tau-equiv} we have $S_1=(0,\hat{\tau})$ and $S_2=[\hat{\tau},\infty)$ (with the understanding that the interval $(0,0)$ is empty).
Substituting these into \eqref{eq:S} yields
\begin{align*}
S = \bigl[\tau_\lambda,\,\min\{\tau_+,\hat{\tau}\}\bigr]
     \;\cup\; \bigl[\max\{\tau_-,\hat{\tau}\},\,\tau_+\bigr],
\label{eq:S-explicit}
\end{align*}
with the convention that an interval is empty whenever its lower bound exceeds its upper bound. This combined with Lemma \ref{lem:tau+tau-} complete the proof.
\end{proof}

Theorem \ref{thm:tau} then follows from Lemmas \ref{lem:monoton-not-strict} and \ref{lem:monotonicity}. We further have the following result.

\begin{corollary}\label{col:left-tau}
Assume that {\rm(IC)$^\infty$} holds and $c_1\in(-c_+^*(0),0)$. Then
\[
c_l(\tau)=
\begin{cases}
c_-^*(\tau),& \tau\in[0,\tau_-],\\
-c_1,& \tau\in[\tau_-,\tau_+],\\
c_+^*(\tau),& \tau\in[\tau_+,+\infty).
\end{cases}
\]
\end{corollary}

\section{Numerical simulations}\label{Sec:Num}

In this section, we perform numerical simulations to verify the explicit spreading speed formulas derived in Theorems \ref{thm:vis-infty} and \ref{thm:tau} for compactly supported initial data \((\text{IC}^\infty)\).

\subsection{Parameter settings}
\noindent

We implement numerical simulations for the nonlocal Nicholson's blowflies model introduced in Section \ref{Sec:Int}. The parameters used in our simulations are summarized in Table \ref{tab:para-blowfly}.
\begin{table}[htbp]
\centering
\caption{A representative parameter set for the nonlocal Nicholson's blowflies equation.}
\label{tab:para-blowfly}
\renewcommand{\arraystretch}{1.2}
\begin{tabular}{l c c}
\toprule
Parameter & Meaning & Suggested value \\
\midrule
\(d\) & Nonlocal dispersal rate & \(1\ \text{day}^{-1}\) \\
\(J(z)\) & Semicircle dispersal kernel & \(\displaystyle J(z) = \frac{2}{\pi L^2}\sqrt{L^2 - z^2}\,\mathbf{1}_{\{|z|\leq L\}}\)$^\dagger$ \\
\(L\) & Support radius of \(J\) & \(2\text{--}3\ \text{km}\) \\
\(\mu\) & Adult death rate & \(0.1\ \text{day}^{-1}\) \\
\(R_+, R_-\) & Far-field linearized birth rates & \(2.0, 1.5\) \\
\(\tau\) & Maturation delay & \(14\ \text{days}\) \\
\(\delta\) & Density-dependence strength & \(0.01\) \\
\bottomrule
\end{tabular}
\begin{flushleft}
\small{$^\dagger$ Here \(J\) is a normalized semicircle kernel supported on \([-L, L]\). For this kernel, the effective diffusion coefficient is \(D_{\text{eff}} = dL^2/8\).}
\end{flushleft}
\end{table}

The delay \(\tau = 14\) days is biologically plausible for blowflies, since the life cycle of \textit{Lucilia cuprina} is reported to take roughly \(11\)--\(21\) days. As \textit{L. cuprina} adults are capable of flying over substantial distances while searching for food, a kilometer-scale daily dispersal parameter is plausible.

\subsection{Verification of the speed selection regimes}
\noindent

To visualize the structure of the speed selection regimes predicted by Theorem \ref{thm:vis-infty}, we first plot the profiles for a representative delay \(\tau = 14\) days in Figure \ref{fig:speed-theory-verify}.

\begin{figure}[htbp]
    \centering
    \includegraphics[width=1\textwidth]{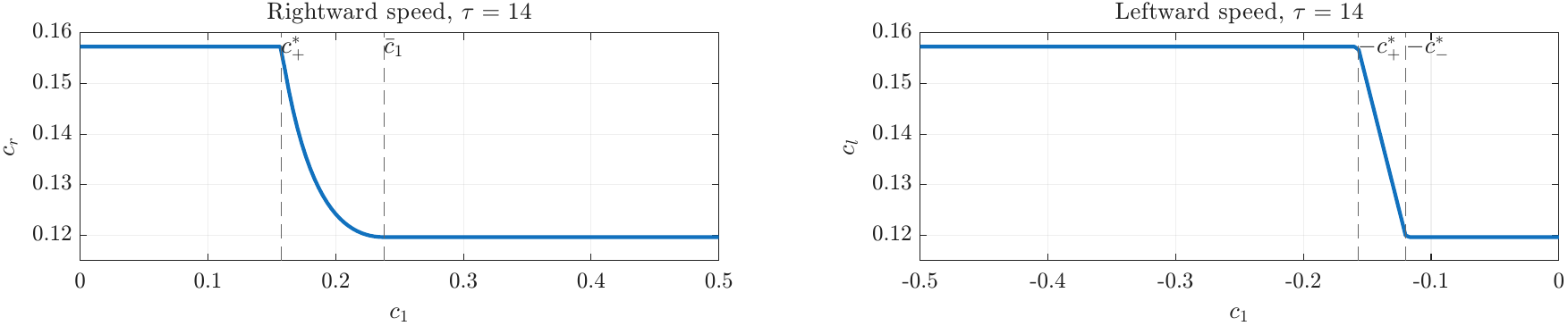}
    \caption{The rightward and leftward spreading speeds as functions of \(c_1\) for a fixed maturation delay \(\tau = 14\) days. }
    \label{fig:speed-theory-verify}
\end{figure}

As illustrated in the left panel of Figure \ref{fig:speed-theory-verify}, the rightward speed consists of three well-defined regimes separated by \(c_+^*\) and \(\bar c_1\):
(i) the rightward far-field intrinsic speed regime for \(c_1 \le c_+^*\),
(ii) the nonlocal selection regime for \(c_+^* < c_1 < \bar c_1\), and
(iii) the leftward far-field intrinsic speed regime for \(c_1 \ge \bar c_1\).
Similarly, the right panel clearly identifies the three subdomains for the leftward speed. The most prominent feature is the tilted segment in the interval \((-c_+^*, -c_-^*)\), which geometrically corresponds to the locking region \(c_l = -c_1\). The domains for \(c_1 \le -c_+^*\) and \(c_1 \ge -c_-^*\) correspond to the intrinsic speed \(c_+^*\) and \(c_-^*\), respectively.

\bigskip

We then examine how the maturation delay affects these regimes. Figure \ref{fig:tau-sweep} displays the numerical profiles of \(c_r\) and \(c_l\) for different values of \(\tau\) belonging to the set \(\{11,12,13, 14\}\).

\begin{figure}[htbp]
    \centering
    \includegraphics[width=1\textwidth]{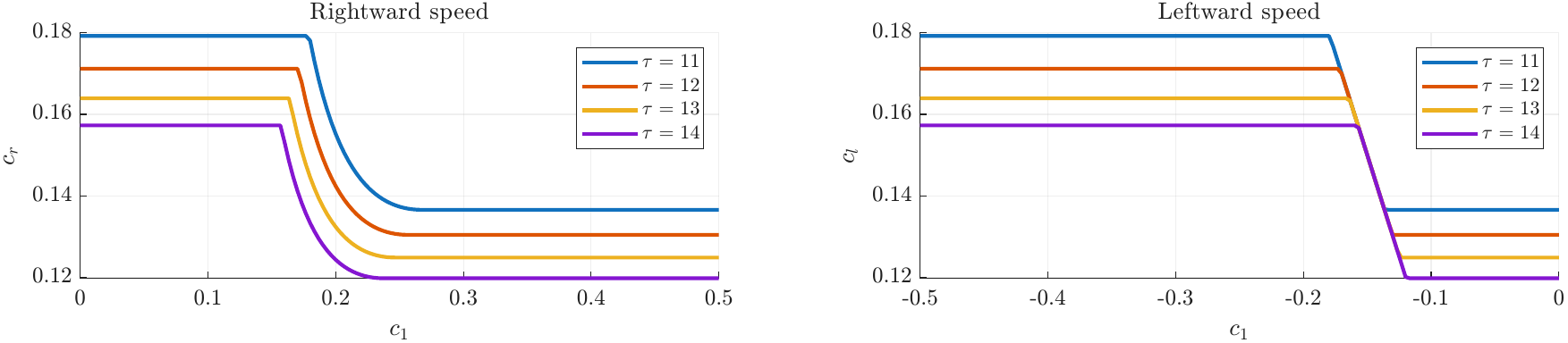}
    \caption{Dependence of the spreading speeds on \(c_1\) and the maturation delay \(\tau\). }
    \label{fig:tau-sweep}
\end{figure}

Consistent with Theorem \ref{thm:tau}, the rightward speed \(c_r\) (left panel) strictly decreases with respect to the delay \(\tau\) for all \(c_1\in\mathbb{R}\). More strikingly, the right panel confirms the delay-independence of the locking phenomenon. Within the interval \(c_1 \in (-c_+^*, -c_-^*)\), the speed satisfies \(c_l = -c_1\), which causes the curves for different \(\tau\) values to collapse onto a single straight line (the diagonal segment) regardless of the delay. This demonstrates the delay-independence predicted by Theorem~\ref{thm:tau}, whose locking window corresponds to \(c_1 \in (-c_+^*(0), 0)\). Outside this locking interval, the monotonic delay-induced deceleration recovers, as evidenced by the clear separation of the curves in the regimes \(c_1 \le -c_+^*\) and \(c_1 \ge -c_-^*\).
These numerical results validate the explicit formulas in Theorem \ref{thm:vis-infty} and the delay locking characterization in Theorem \ref{thm:tau}.

\appendix
\section{The proof of Lemma \ref{lem:w^vareps-bound-2} for Laplacian diffusion}\label{App:Local}

In this section, we prove Lemma \ref{lem:w^vareps-bound-2} for the Laplacian diffusion case. By the hyperbolic scaling and the WKB transformation, $w^\varepsilon$ associated with Laplacian diffusion satisfies that
\begin{align}\label{w^varepsilon-local}
-\varepsilon d\partial_{xx}w^\varepsilon + d|\partial_x w^\varepsilon|^2 + \partial_tw^\varepsilon -\mu  + \mu \frac{f^\varepsilon(x-c_1t,u^\varepsilon(t-\varepsilon\tau,x))}{u^\varepsilon(t,x)} =0\quad\text{ in } (0,\infty)\times\mathbb{R}.
\end{align}

\begin{lemma}\label{lem:w-vareps-bound}
Assume that {\rm (F1)--(F4)} hold. Let $w^\varepsilon$ be a solution of \eqref{w^varepsilon-local} with initial condition {\rm(\text{IC}$^\infty)$.} Then for any compact subset $Q_1$ of $(0,\infty)\times\mathbb{R}$, there exists a positive constant $C=C(Q_1)$ and $0<\varepsilon_0<1$ such that for any $\varepsilon\in(0,\varepsilon_0)$, $\sup_{Q_1}|w^\varepsilon|\leq C.$
\end{lemma}

\begin{proof}
We only prove the upper bound. It follows from (F1) that $u$ satisfies $u_t\ge du_{xx}-\mu u$.
By the comparison principle, $u(t,x)\ge e^{-\mu t}\int_{\mathbb{R}}G_{d}(t,x-z)u_0(z)\,dz,$
where $G_{d}(t,\xi)=\frac{1}{\sqrt{4\pi dt}}e^{-\frac{\xi^2}{4dt}}$.

Let $Q=[\iota,T]\times[-R,R]$ with $\iota>0$.
Since $u_0$ satisfies {\rm(\text{IC}$^\infty)$}, there exist a closed interval $I=[a,b]$ and a constant $m>0$ such that $u_0\ge m$ on $I$.
For $(t,x)\in Q$ and $\varepsilon\le1$, we have
\[
\begin{aligned}
u^\varepsilon(t,x) &= u\Bigl(\frac{t}{\varepsilon},\frac{x}{\varepsilon}\Bigr)
\ge e^{-\mu T/\varepsilon}\int_{I}
\frac{1}{\sqrt{4\pi d t/\varepsilon}}
\exp\!\Bigl(-\frac{(x/\varepsilon-z)^2}{4dt/\varepsilon}\Bigr)\,m\,dz \\
&= m e^{-\mu T/\varepsilon}\sqrt{\frac{\varepsilon}{4\pi d t}}
\int_{a}^{b}\exp\!\Bigl(-\frac{(x-\varepsilon z)^2}{4dt\varepsilon}\Bigr)dz .
\end{aligned}
\]
On $Q$, $|x|\le R$, $t\ge\iota$, so for $z\in I$ and $\varepsilon\le1$,
\[
|x-\varepsilon z|\le R+\varepsilon\max\{|a|,|b|\}\le R+\max\{|a|,|b|\}=:R_1.
\]
Hence
\[
u^\varepsilon(t,x)\ge m\,e^{-\mu T/\varepsilon}\sqrt{\frac{\varepsilon}{4\pi d T}}
\,(b-a)\,e^{-R_1^2/(4d\iota\varepsilon)}
=  C_1\sqrt{\varepsilon}\,e^{- C_2/\varepsilon},
\]
where $C_1=\frac{m(b-a)}{\sqrt{4\pi d T}}$, and $C_2=\mu T+\frac{R_1^2}{4d\iota}$.
We now revert to $w^\varepsilon=-\varepsilon\ln u^\varepsilon$ that
\[
w^\varepsilon(t,x)\le -\varepsilon\ln  C_1-\frac{\varepsilon}{2}\ln\varepsilon+ C_2.
\]
Since $-\varepsilon\ln  C_1-\frac{\varepsilon}{2}\ln\varepsilon\to0$ as $\varepsilon\to0$, there exists $0<\varepsilon_0<1$ such that for all $\varepsilon\in(0,\varepsilon_0)$ and all $(t,x)\in Q$,
\[
w^\varepsilon(t,x)\le  C_2+1,
\]
which completes the proof.
\end{proof}

\section*{Acknowledgments}
\noindent

The research of W.-T. Li was partially supported by NSF of China (Nos. 12531008; 12271226). The research of K.-Y. Lam was partially supported by the Hong Kong Research Grants Council (HK-RGC Grant 15305824).


\end{document}